\documentclass[12pt,reqno]{amsart}
\usepackage{amssymb,mathrsfs,longtable,cite}

\newcommand{\D}{\mathcal{D}}

\newcommand{\F}{\mathcal{F}}

\newcommand{\X}{\mathcal{X}}

\newcommand{\cS}{\mathcal{S}}
\newcommand{\cH}{\mathcal{H}}

\newcommand{\cN}{\mathcal{N}}

\newcommand{\C}{\mathbb{C}}
\newcommand{\N}{\mathbb{N}}
\newcommand{\R}{\mathbb{R}}
\newcommand{\Z}{\mathbb{Z}}

\newcommand{\bE}{\mathbb{E}}
\newcommand{\bM}{\mathbb{M}}
\newcommand{\bK}{\mathbb{K}}
\newcommand{\bG}{\mathbb{G}}
\newcommand{\bH}{\mathbb{H}}
\newcommand{\bI}{\mathbb{I}}

\newcommand{\sI}{\mathscr{I}}

\newcommand{\sX}{\mathscr{X}}

\newcommand{\sg}{\mathfrak{s}}

\newcommand{\Soli}{\mathscr{S}}
\newcommand{\SBC}{\operatorname{SBC}}
\newcommand{\Stz}{\operatorname{Stz}}
\newcommand{\ST}{\mathfrak{st}}
\newcommand{\ml}[1]{\lceil #1 \rfloor}

\newcommand{\al}{\alpha}
\newcommand{\be}{\beta}
\newcommand{\ga}{\gamma}
\newcommand{\de}{\delta}
\newcommand{\e}{\varepsilon}
\newcommand{\fy}{\varphi}
\newcommand{\om}{\omega}
\newcommand{\la}{\lambda}
\newcommand{\te}{\theta}
\newcommand{\s}{\sigma}
\newcommand{\ta}{\tau}
\newcommand{\ka}{\kappa}
\newcommand{\x}{\xi}
\newcommand{\y}{\eta}
\newcommand{\z}{\zeta}

\newcommand{\De}{\Delta}
\newcommand{\Om}{\Omega}
\newcommand{\Ga}{\Gamma}
\newcommand{\La}{\Lambda}

\newcommand{\p}{\partial}
\newcommand{\na}{\nabla}
\newcommand{\Cu}{\bigcup}
\newcommand{\Ca}{\bigcap}
\newcommand{\re}{\mathop{\mathrm{Re}}}

\newcommand{\weak}{\operatorname{w-}}
\newcommand{\weakto}{\rightharpoonup}

\newcommand{\sign}{\operatorname{sign}}

\newcommand{\lec}{\lesssim}
\newcommand{\gec}{\gtrsim}

\newcommand{\IN}[1]{\text{ in }#1}

\newcommand{\etc}{,\ldots,}
\newcommand{\I}{\infty}

\newcommand{\follows}{\Longleftarrow}

\newcommand{\ti}{\widetilde}
\newcommand{\ba}{\overline}

\newcommand{\U}{\underline}
\newcommand{\LR}[1]{{\langle #1 \rangle}}
\newcommand{\Lim}{\lim\limits}

\newcommand{\diff}[1]{{\triangleleft #1}}
\newcommand{\pa}{\triangleright}

\newcommand{\EQ}[1]{\begin{equation}\begin{split} #1 \end{split}\end{equation}}

\newcommand{\BR}[1]{\left\{#1\right\}}

\newcommand{\tand}{\ \text{ and }\ }

\newcommand{\Del}[1]{}
\newcommand{\CAS}[1]{\begin{cases} #1 \end{cases}}

\newcommand{\pt}{&}
\newcommand{\pr}{\\ &}
\newcommand{\pq}{\quad}
\newcommand{\pn}{}
\newcommand{\prq}{\\ &\quad}
\newcommand{\prQ}{\\ &\qquad}

\numberwithin{equation}{section}

\newtheorem{thm}{Theorem}[section]

\newtheorem{lem}[thm]{Lemma}
\newtheorem{prop}[thm]{Proposition}

\theoremstyle{remark}
\newtheorem{rem}{Remark}

\newcommand{\GS}{\operatorname{GS}}

\newcommand{\II}{_{(\I)}}

\newcommand{\pp}{\mathfrak{p}}
\newcommand{\sE}{\mathscr{E}}
\newcommand{\Sup}{\operatorname{Sup}}

\begin{document}

\title[Dynamics for NLS with a potential]{Global dynamics below excited solitons \\ for the nonlinear Schr\"odinger equation \\ with a potential}
\begin{abstract}
Consider the nonlinear Schr\"odinger equation (NLS) with a potential with a single negative eigenvalue. 
It has solitons with negative small energy, which are asymptotically stable, and, if the nonlinearity is focusing, then also solitons with positive large energy, which are unstable. 
In this paper we classify the global dynamics below the second lowest energy of solitons under small mass and radial symmetry constraints. 
\end{abstract}

\author[K.~Nakanishi]{Kenji Nakanishi}

\address{Department of Pure and Applied Mathematics
Graduate School of Information Science and Technology
Osaka University, Suita, Osaka 565-0871, JAPAN}

\email{nakanishi@ist.osaka-u.ac.jp}


\maketitle

\tableofcontents

\section{Introduction}
\subsection{Background and motivation}
Nonlinear dispersive equations have solutions with various types of behavior in time, typically {\it scattering} (globally dispersive), {\it blow-up}, and solitary waves, i.e., {\it solitons}. 
In the recent years, especially since the work of Kenig and Merle \cite{KM}, global dynamics leading to those different types have been revealed among large general solutions, so that one can partially predict evolution of each solution from the initial data. 
Kenig and Merle \cite{KM} studied  the energy-critical NLS 
\EQ{
 i\dot u - \De u = |u|^4u, \pq u(t,x):\R^{1+3}\to\C,}
and proved that all solutions with energy less than the ground state $W$
\EQ{
 \pt E(u)=\int_{\R^3}\frac{|\na u|^2}{2}-\frac{|u|^6}{6}dx<E(W),
 \pr W(x):=(1+|x|^2/3)^{-1/2},\pq -\De W = W^5,}
either scatter or blow-up, and that the two types of behavior are distinguished by some explicit functionals of the initial data. For example, 
\EQ{
 K(u(0))=\int_{\R^3}|\na u(0)|^2-|u(0)|^6dx \CAS{\ge 0 \implies \text{scattering,} \\ <0 \implies \text{blow-up.}} }
The distinction is essentially the same as that in the classical result for the nonlinear Klein-Gordon equation by Payne and Sattinger \cite{PS} into global existence vs.~blow-up, but the crucial aspect of Kenig-Merle's work is to reveal and exploit the global dispersion in the scattering part. 
It was extended to the threshold energy $E(u)\le E(W)$ by Duyckaerts and Merle \cite{DM}, and then slightly above the ground state by Schlag and the author \cite{NSK}, for the nonlinear Klein-Gordon equation
\EQ{ \label{NLKG}
 \pt \ddot u - \De u + u = u^3, \pq u(t,x):\R^{1+3}\to\R,
 \pr E(u)=\int_{\R^3}\frac{|\dot u|^2+|\na u|^2+|u|^2}{2}-\frac{|u|^4}{4}dx<E(Q)+\e^2,}
where $Q\in H^2(\R^3)$ is the unique positive radial solution or the ground state of 
\EQ{ \label{eq Q}
  -\De Q+Q=Q^3.}
The types of behavior in that case are separated into 9 sets of solutions by center-stable and center-unstable manifolds of the ground state, and the mechanism of transition between scattering and blow-up is revealed. 
Furthermore, Duyckaerts, Kenig and Merle \cite{DKM} established a complete classification of asymptotic behavior of solutions, for the energy-critical wave equation
\EQ{
 \ddot u - \De u = u^5, \pq u(t,x):\R^{1+3}\to\R,}
in terms of resolution into solitons (i.e.~rescaled $W$), without any size restriction on the initial data. All of these works have been extended to several equations and settings including the above examples, except the soliton resolution which is yet limited to variants of energy-critical wave equations. 

General dynamics are, however, far more complicated for more general or physical equations. In particular, many equations, especially of the NLS type, have many solitons, differing in shape, energy, stability, etc. Heuristically, unstable solitons are expected to collapse into stable solitons, radiating dispersive waves. 
For small solitons of the NLS with a decaying potential $V(x)$, Tsai and Yau \cite{TY1,TY2,TY3,TY4} first proved such a phenomenon, as well as asymptotic stability of the ground soliton, in the case $-\De+V$ has two well-positioned negative eigenvalues. Since then, there have been intensive studies (cf.~\cite{TS,SW,GS,NPT,CM}) on global behavior of small solutions including many solitons, but very little is rigorously known about dynamical relation between solitons which are neither close nor similar to each other. It seems hard in such cases to construct or control solutions in a precise way along some anticipated evolution. A more natural strategy is to deal altogether with general solutions including or at least close to those solitons, with less precise information on individual trajectories. 

\subsection{Setting and the main result}
As a first step toward the above problem, we consider the NLS with a potential 
\EQ{ \label{NLSP}
 \pt i\dot u + H u = \sg|u|^2u, \pq H:=-\De+V, \pq \sg=\pm,
 \prq u(t,x):\R^{1+3}\to\C, \pq V(x):\R^3\to\R,}
in the simplest non-trivial setting, namely the case with the unique eigenvalue 
\EQ{ \label{def phi0}
 \pt H\phi_0 = e_0\phi_0,\pq e_0<0, \pq 0<\phi_0\in H^2(\R^3), \pq \|\phi_0\|_{L^2(\R^3)}=1,}
with $spec(H|_{\phi_0^\perp})=[0,\I)$ absolutely continuous, and the radial symmetry restriction 
\EQ{
 u(t,x)=u(t,|x|), \pq V(x)=V(|x|).}
Hence the initial data set is the radial subspace of the Sobolev space
\EQ{ \label{def H1r}
 H^1_r(\R^3):=\{u\in L^2(\R^3)\mid \na u\in L^2(\R^3),\ u(x)=u(|x|)\}.}
The nonlinearity can be either defocusing $\sg=-$ or focusing $\sg=+$. 
In the focusing case $\sg=+$, the above equation is one of the simplest equations with both stable and unstable solitons, where the former is small and the latter is large. 
The goal of this study is a complete description of global dynamics in a fairly large solution space, containing both the stable and the unstable solitons. 
In this paper, we consider the region of small mass and an upper energy constraint which eliminates the unstable solitons. An implication of the main result is that if an unstable (large) soliton with small mass and the second largest energy is perturbed to decrease its energy and the mass, then it either blows up or collapses into a (small) ground state soliton, radiating most of the energy (which is large) into a dispersive wave. The two types of behavior is distinguished by a functional of the initial data, similarly to Kenig-Merle or Payne-Sattinger. 

In order to state the main result, we need a few more assumptions on $V$. A simple sufficient condition is that $V$ is in the Schwartz class and $H$ has no resonance: 
\EQ{
 \dot H^1(\R^3)\ni\fy,\pq (-\De+V)\fy=0 \implies \fy=0.}
The existence of small solitons is well known in the above setting. 
The function $u(t,x)=e^{-it\om}\fy(x)$ is a solution of \eqref{NLSP} iff 
\EQ{ \label{sNLSP}
 (H+\om)\fy=\sg|\fy|^2\fy.}
In this paper we call a solution of \eqref{sNLSP} a soliton, denoting the set of solitons by 
\EQ{ \label{def Soli}
 \Soli:=\{\fy\in H^1_r(\R^3) \mid \exists\om\in\R\ s.t.\ \eqref{sNLSP}\}} 
and the energy (Hamiltonian) and the mass (charge) by 
\EQ{ \label{def EM}
 \bE(u):=\int_{\R^3} \frac{|\na u|^2+V|u|^2}{2}-\frac{\sg|u|^4}{4}dx,
 \pq \bM(u):=\int_{\R^3} \frac{|u|^2}{2}dx,}
which are continuous on $H^1(\R^3)$ and conserved for \eqref{NLSP}. 
For each fixed mass $\bM(\fy)=\mu>0$, we can define the energy levels of solitons by induction on $j=0,1,2\etc$
\EQ{ \label{def Ej}
 \sE_j(\mu):=\inf\{\bE(\fy)\mid \fy\in\Soli,\ \bM(\fy)=\mu,\ \bE(\fy)>\sE_{j-1}(\mu)\},}
where $\sE_{-1}(\mu):=-\I$ and $\inf\emptyset:=\I$, then classify the solitons 
\EQ{ \label{def Solij}
 \Soli_j:=\{\fy\in\Soli \mid \bE(\fy)=\sE_j(\bM(\fy))\}.}
$\Soli_0$ is the set of least energy solitons, namely {\it the ground states}, while $\Soli_j$ is the $j$-th {\it excited state} for $j\ge 1$. 
In this paper, we are concerned only with $\Soli_0$ and $\Soli_1$. 

It is easy to observe that the ground states for small mass are bifurcation from $0$ generated by the linear ground state $\phi_0$ in \eqref{def phi0}. 
More precisely, there exists $0<z_b\ll 1$ and a $C^1$ map 
\EQ{ 
 (\Phi,\Om):D_b:=\{z\in\C\mid |z|^2<2z_b\} \to H^1_r(\R^3)\times\R}
such that $(\fy,\om)=(\Phi[z],\Om[z])$ solves \eqref{sNLSP} for each $z\in D_b$ and 
\EQ{
 \Phi[z] = z\phi_0 + \ga, \pq \ga \perp \phi_0, \pq \|\ga\|_{H^1}\lec|z|^3.} 
See \cite{gnt} for a proof in a more general setting. 
We can prove that $\Phi(D_b)=\Soli_0$ under the small mass constraint $\bM<z_b^2$, while the first excited energy satisfies 
\EQ{
 \sE_1(\mu)=\CAS{\bE^0(Q)\bM(Q)\mu^{-1}(1+o(1))  &(\sg=+) \\ \I &(\sg=-),}}
as $\mu\to 0$, where $\bE^0$ denotes the energy without the potential, namely
\EQ{ \label{def E0}
 \bE^0(\fy):=\int_{\R^3}\frac{|\na u|^2}{2}-\sg\frac{|u|^4}{4}dx.}
In fact, in the defocusing case $\sg=-$, the soliton \eqref{sNLSP} is unique for each fixed $\bM(\fy)=\mu>0$ modulo the gauge symmetry $e^{i\te}$. 
In the focusing case $\sg=+$, the first excited states are generated by scaling of $Q$ 
\EQ{
 \Soli_1\ni\fy=\om^{1/2}(Q+o(1))(\om^{1/2}x) \pq(\bM(\fy)\to 0)}
We do not need the above characterizations of $\Soli_1$, but the variational property with respect to the virial-type functional
\EQ{ \label{def K2}
  \bK_2(u):=\int_{\R^3}|\na u|^2-\frac{rV_r|u|^2}{2}-\sg\frac{3|u|^4}{4}dx=\p_{\al=1}\bE(\al^{3/2}u(\al x))}
plays a crucial role as in the case $V=0$. Henceforth $\p_{\al=a}$ denotes the partial derivative with respect to $\al$ at $\al=a$, namely 
\EQ{
 \p_{\al=a}f :=\lim_{\e\to 0}\frac{f|_{\al=a+\e}-f|_{\al=a}}{\e}.}
The following is the main result of this paper.
\begin{thm} \label{main} 
There exists $0<\mu_\star\ll 1$ such that for any $u(0)\in H^1_r(\R^3)$ satisfying $\bM(u(0))\le\mu_\star$ and $\bE(u(0))<\sE_1(\bM(u(0)))$, the corresponding solution of \eqref{NLSP} either blows up in finite time both in $t>0$ and in $t<0$, or scatters as $t\to\pm\I$ to the ground states $\Soli_0$. More precisely, in the former case, there are $T_\pm\in(0,\I)$ such that the unique solution $u\in C((-T_-,T_+);H^1_r)$ exists and 
\EQ{
 \lim_{t\to\pm(T_\pm-0)}\|\na u(t)\|_{L^2(\R^3)}=\I=\limsup_{t\to\pm(T_\pm-0)}\|u(t)\|_{L^\I(\R^3)}.}
In the latter case, there are a $C^1$ function $z:\R\to D_b\subset\C$ and $u_\pm\in H^1_r(\R^3)$ such that $|z(t)|$ converges as $t\to\pm\I$ and 
\EQ{ \label{scatt to S_0}
 \lim_{t\to\pm\I}\|u(t)-\Phi[z(t)]-e^{-it\De}u_\pm\|_{H^1(\R^3)}=0.}
Moreover, the blow up occurs if and only if 
\EQ{ \label{bup cond}
 \sg=+, \pq \|\na u(0)\|_{L^2(\R^3)}>1, \text{ and } \bK_2(u(0))<0,}
which persists in $t$ as long as the solution $u$ exists. 
\end{thm}

The above theorem contains the asymptotic stability of the ground state $\Soli_0$ for small $H^1$ radial solutions. This part is contained in the asymptotic stability for small solutions in \cite{gnt} by Gustafson, Tsai and the author, which does not need the radial symmetry restriction. 

If the potential $V=0$, then there is no small soliton such as $\Soli_0$, but the ground state $Q$ as in \eqref{NLKG} exists and unstable. 
In that case, the above result regarding $\Soli_0=\{0\}$ and $\Soli_1=\{\al Q(\al x)\}_{\al>0}$ was obtained by Holmer and Roudenko \cite{HR}, extended to the non-radial case by Duyckaerts, Holmer and Roudenko \cite{DHR}, to the threshold energy by Duyckaerts and Roudenko \cite{DR}, and slightly above the threshold (in the radial case) by Schlag and the author \cite{NSS}. 
In these works there is no small-mass constraint as above, but it is not an essential difference, because the scale invariance in the case $V=0$ allows one freely to add or remove such a restriction. 

\subsection{Difficulties and ideas in the proof} \label{ss:diff}
The proof follows the strategy of Kenig and Merle \cite{KM}, which consists of a stationary part based on the classical variational argument for the elliptic equation \eqref{sNLSP}, and a dynamical (or scattering) part based on the variational argument in space-time: the profile decomposition by Bahouri and G\'erard \cite{BG}. 

The problem caused by the potential in the stationary variational argument can be read immediately from the virial identity 
\EQ{
 \p_t\LR{i\dot u|x\cdot\na u}=-2\bK_2(u).}
In the absence of $V$, the functional $\bK_2$ can not vanish under the energy constraint except at $0$, and so sign-definite along each trajectory. This leads to monotonicity in the virial identity, which has been the crucial starting point for $V=0$, including the case slightly above the ground state \cite{NSK}, where possible change of $\sign\bK_2$ was controlled by using the linearized operator around $Q$. 

In the presence of $V$, the functional $\bK_2$ changes the sign around the ground solitons $\Soli_0$. Note that this problem does not arise in the elliptic equation \eqref{sNLSP} using the Nehari functional 
\EQ{
 \p_{\al=1}(\bE+\om\bM)(\al u)=\int_{\R^3}|\na u|^2+(V+\om)|u|^2-|u|^4dx,}
because the excited states $\Soli_1$ can be distinguished from the ground states $\Soli_0$ by the time frequency $\om$. Indeed, $\om\to-e_0$ on $\Soli_0$ while $\om\to\I$ on $\Soli_1$ as $\bM\to 0$. 
In contrast, the virial functional $\bK_2$ is independent of $\om$, since it is derived by the $L^2$-preserving dilation. 

The above problem in the virial identity is however easily solved using the fact that the disturbance of $\sign\bK_2$ occurs only in a small neighborhood of $H^1(\R^3)$, where we have the asymptotic stability of $\Soli_0$ from \cite{gnt}. 
In fact, the region $\bK_2\lec \bM\ll 1$ splits into two sets far from each other in $H^1(\R^3)$: one around $0$ satisfying 
\EQ{
 \|\na\fy\|_{L^2(\R^3)}^2+\|\fy\|_{L^4(\R^3)}^2 \lec \bM(\fy) \ll 1,}
and the other with large energy satisfying 
\EQ{
 \min(\|\na\fy\|_{L^2(\R^3)}^2,\|\fy\|_{L^4(\R^3)}^4) \gec \bM(\fy)^{-1} \gg 1.}See Lemma \ref{lem:SMD} for a more general statement with a proof. 
In \eqref{bup cond}, the condition $\|\na u\|_{L^2}>1$ is imposed only to distinguish the above two cases, so there are many alternative conditions, such as $\|u\|_{L^4}>1$. 

The problems in the space-time variational argument, caused by the potential, or more precisely by the stable solitons $\Soli_0$, appear more fundamental. 
First, we should obviously remove the stable soliton part from the solution to apply the profile decomposition, as it aims at global dispersion or space-time integrability of the solutions. 
Second, the linear terms of the dispersive part, namely the interaction with the small soliton, can not be treated as part of the nonlinear perturbation, since it would require smallness in $L^2_t$ of the remainder of the profile decomposition, which is impossible as long as we take the initial data from the $L^2_x$ Sobolev space. 

Therefore, we have to consider the linearized equation around the small soliton as the reference equation in the profile decomposition for the dispersive component. 
Since the modulation in time, namely $z(t)$ in \eqref{scatt to S_0}, depends on the solution, it means that we have to consider a sequence of equations corresponding to the sequence of solutions to which we apply the concentration compactness. 

Another problem is that we have very poor control on the global or asymptotic behavior of $z(t)$. 
For example, the convergence of $|z(t)|$ as $t\to\I$ becomes arbitrarily slow by choosing small $H^1$ data, see \cite[Theorem 1.9]{gnt}. 
This causes difficulties at least in the following two places. 

First, the nonlinear profile decomposition is a method to approximate solutions globally in time, but we can not do it for the soliton part $z(t)$. 
Therefore we have to distinguish time into two regimes: around and away from the profiles, approximating $z$ only in the former, while relying on the smallness of the dispersive component in the latter. 

Second, the nonlinear profiles moving to $t\to\pm\I$ were defined in Kenig-Merle \cite{KM} by the wave operator, i.e., solving the final state problem with the linear profile as the scattering state. 
The existence of solution to the final state problem in the current setting, namely around the ground states $\Soli_0$, was proved in \cite{gnt}, but we do not even know the uniqueness, while we would need some continuity estimate. 
Hence we have to define the nonlinear profiles in another way, that is the weak limit along a time sequence, proving afterward that the linear profile is the scattering state. 
The drawback of this definition is that we can not construct global approximation at one stroke as in Kenig-Merle, but have to proceed step by step over each profile. 
The approach in this paper can be roughly regarded as a hybrid between Bahouri-G\'erard \cite{BG} and Kenig-Merle \cite{KM}. The former used the scattering to describe the limit of sequence of solutions, while the latter used the limit of sequence to obtain the scattering. We need to proceed from both the sides. 

Yet another complication in the estimates is due to the quadratic nonlinearity in the equation after linearization, to which we can not directly apply the Strichartz estimate to obtain Lipschitz estimate in the energy space for global perturbation, together with the smallness of the remainder in the profile decomposition. 
To solve this problem, we follow the idea in \cite{CNLKG}, using non-admissible Strichartz norms and measuring the initial data by the Strichartz norm. 
Such estimates are derived for the linearized equation, treating the time dependent potential by the double endpoint Strichartz estimate as in \cite{gnt}, but it requires the non-admissible version, obtained independently by Foschi \cite{F} and by Vilela \cite{V}. 

Extension of the result in this paper to the lower space dimensions would require similar modification to the argument by Mizumachi \cite{M1,M2}, who extended the small data result of \cite{gnt} by replacing the endpoint Strichartz estimate with Kato's weighted $L^2$ space-time estimate. 
Apart from that issue, it should be rather straightforward to extend it to general space dimensions and general power nonlinearity between the mass and the energy critical exponents, namely 
\EQ{
 \pt i\dot u + H u = \sg|u|^\al u, \pq u(t,x):\R^{1+d}\to\C, \pq \frac{4}{d}<\al<\frac{4}{d-2},}
even though the 3D-cubic setting is exploited for minor simplification in several places of this paper. 

\subsection{Notation} \label{ss:nota}
$L^p_x$, $B^s_{p,q}$, and $H^s_p$ denote respectively the standard Lebesgue, inhomogeneous Besov, and inhomogeneous Sobolev spaces on $\R^3$. 
The $L^2$ based Sobolev space is denoted by $H^s=H^s_2$. 
The $L^p_x$ norm is often denoted by $\|\cdot\|_p$. For any function space $X$ on $\R^3$, the subspace of radial functions is denoted by $X_r$, and the $L^p$ space in $t\in\R$ with values in $X$ is denoted by $L^p_tX$. 
For any function space $Z$ on $\R^{1+3}$ and $I\subset\R$, $Z(I)$ denotes the restriction onto $I\times\R^3$. The $L^2$ inner products on $\R^3$ are denoted by 
\EQ{ \label{def inner}
 (f|g):=\int_{\R^3}f(x)\ba{g(x)}dx, \pq \LR{f|g}:=\re(f|g).}

\subsection{Assumptions on $V$} \label{ss:asm V} 
Let $L^\I_0(\R^3)=\{\fy \in L^\I(\R^3) \mid \Lim_{R\to\I}\|\fy\|_{L^\I(|x|>R)}=0\}$. 
The precise assumption on $V$ is as follows. 
\begin{enumerate}
\item $V:\R^3\to\R$ is radially symmetric. 
\item $V,\ x\na V,\ x^2\na^2 V\in (L^2+L^\I_0)(\R^3)$ and $V/|x| \in L^1(\R^3)$. 
\item $-\De+V$ on $L^2_r(\R^3)$ has a unique and negative eigenvalue. 
\item The wave operator $W=\Lim_{t\to\I}e^{itH}e^{it\De}$ and its adjoint $W^*$ are bounded on the Sobolev space $W^{k,\pp}(\R^3)$ for some $\pp>6$ and $k=0,1$. 
\end{enumerate}
The above assumption (ii) implies that 
\EQ{ \label{V decay}
 \lim_{|x|\to\I} |V(x)|+|x\na V(x)|= 0} 
by the radial Sobolev inequalities, cf.~Appendix \ref{app:dec V}. 
By Beceanu \cite[Corollary 1.5]{B}, the assumption (iv) is fulfilled if $0$ is neither an eigenvalue nor resonance of $H$, and $V\in (L^\pp \cap \F\dot B^{1/2}_{2,1})(\R^3)$. 

For example, if $\psi\in\cS(\R^3)$ is a radial positive function, the above assumptions (i)-(iv) are satisfied by $V=-a\psi$ for $a\in (1/a_1,1/a_2)$, where $a_1>a_2>0$ are the largest and the second largest eigenvalues of the compact self-adjoint operator $(-\De)^{-1/2}\psi(-\De)^{-1/2}$ on $L^2_r(\R^3)$.

\section*{Acknowledgements}
This work was originally started from intensive discussions with Stephen Gustafson and Tai-Peng Tsai. The author would like to thank them for useful comments on the manuscript. 
He is also grateful for Scipio Cuccagna, Masaya Maeda, Yoshio Tsutsumi, and the anonymous referee for their comments and pointing out some errors and missing references in the first version. 

\section{Standing waves}
This section collects some properties of the solutions of \eqref{sNLSP}, namely solitons. 
It is easy to see $\om>0$ for $\fy\in H^1_r(\R^3)$, using the asymptotic behavior of the ODE as $r=|x|\to\I$. 
We will see that in the defocusing case $\sg=-$, there is a unique soliton $\fy$ for each $\om\in(0,-e_0)$ and nothing else. In the focusing case $\sg=+$, there is a soliton for each $\om\in(-e_0,\I)$, among which we can specify the ground state and the first excited state for each fixed small mass under the radial constraint. 

\subsection{Energy functionals}
For any $\fy\in H^1(\R^3)$, and $V:\R^3\to\R$, we define the following functionals on $H^1(\R^3)$.  
\EQ{ \label{def funct}
 \pt \ml{V}(\fy):=\LR{V\fy|\fy}/2, \pq \bM(\fy):=\ml{1}(\fy)=\|\fy\|_2^2/2, 
  \pq \bG(\fy):=\sg\|\fy\|^4_4/4, 
 \pr \bH(\fy):=\LR{H\fy|\fy}/2=\|\na\fy\|_2^2/2+\ml{V}(\fy),
   \pq \bE:=\bH-\bG.}
The energy $\bE$ and the mass $\bM$ are conserved in time for \eqref{NLSP}. 
The corresponding quantities without the potential $V$ are denoted by $\bH^0$, $\bE^0$, etc. 
\EQ{ \label{def H0}
  \pt \bH^0(\fy):=\|\na\fy\|_2^2/2, \pq \bE^0(\fy):=\|\na\fy\|_2^2/2-\sg\|\fy\|_4^4/4, \pr \bK_2^0(\fy):=\p_{\al=1}\bE^0(\al^{3/2}\fy(\al x)).}
For the variational property in the focusing case, we need the dilation operator 
\EQ{ \label{def Stp}
 \cS^t_p\fy(x):=e^{3t/p}\fy(e^tx), \pq \cS'_p\fy(x)=(x\cdot\na+3/p)\fy(x),}
which preserves the $L^p_x$ norm. 
The same notation is used for the functional 
\EQ{ \label{def S'}
 (\cS'_p F)(\fy):=\p_{t=0}F(\cS^t_p\fy).}
Then we have 
\EQ{
 \pt \cS'_p\bM=(6/p-3)\bM, \pq \cS'_p\bH^0=(6/p-1)\bH^0, 
 \pr \cS'_p\bG=(12/p-3)\bG,
 \pq \cS'_p\ml{V}=-\ml{\cS'_{p/(p-2)}V}.}
The $L^2$-scaling derivative plays a crucial role via the virial identity
\EQ{
 \bK_2 :=\cS'_2\bE = 2\bH^0-3\bG-\ml{\cS'_\I V}.}
The following functional is used for convexity of $\bE$ in $\cS^t_p$:
\EQ{  \label{def I}
 \bI:=\bE-\bK_2/2=\bG/2+\ml{\cS'_{3/2}V}/2.}

If there is a family of solitons $\om\mapsto\fy_\om\in H^1$ differentiable in $\om\in\R$, then we have 
\EQ{ \label{E-M diff}
 \p_\om \bE(\fy_\om) = \LR{\bE'(\fy_\om)|\fy_\om'} = \LR{-\om\fy_\om|\fy'_\om} = -\om \p_\om\bM(\fy_\om).}
For the potential part, we will frequently use the following bound
\begin{lem} \label{decop V}
Let $V\in(L^2+L^\I_0)(\R^3)$. Then for any $\e>0$, there is $C>0$ such that
\EQ{ \label{est decop V}
 H^1(\R^3)\ni\forall\fy,\pq |\ml{V}(\fy)| \le \min(\e\|\fy\|_4^2+C\|\fy\|_2^2,\e\|\fy\|_2^2+C\|\fy\|_4^2).}
\end{lem}
\begin{proof}
Let $V=V_2+V_\I$ where $V_2\in L^2$ and $V_\I\in L^\I_0$. 
For any $h>0$, we have 
\EQ{
 |\ml{V_2}(\fy)| \le \|V_2\|_{L^2(V>h)}\|\fy\|_4^2+h\|\fy\|_2^2,}
where  $\|V_2\|_{L^2(V>h)}\to 0$ as $h\to\I$, so the right hand side is in the form of \eqref{est decop V}, choosing $h>0$ such that $\|V_2\|_{L^2(V>h)}\le\e$ or $h\le\e$. 
For any $R>0$, 
\EQ{
 |\ml{V_\I}(\fy)| \lec \|V_\I\|_{L^\I(|x|<R)}\|\fy\|_4^2R^{3/2}+\|V_\I\|_{L^\I(|x|>R)}\|\fy\|_2^2,}
where $\|V_\I\|_{L^\I(|x|>R)}\to 0$ as $R\to\I$, so the right hand side is also in the form of \eqref{est decop V}, choosing $R>0$ such that $\|V_\I\|_{\I}R^{3/2}\le\e$ or $\|V_\I\|_{L^\I(|x|>R)}\le\e$. 
Adding the above two estimates yields the conclusion. 
\end{proof}

\subsection{Small solitons}
For small mass, the ground state is the bifurcation from zero, generated by the ground state $\phi_0$ of $H$. 
The following precise statement can be extracted from \cite[Lemma 2.1]{gnt}
\begin{lem} \label{lem:Phi}
There exists $0<z_b\ll 1$ and a $C^1$ map 
\EQ{
 (\Phi,\Om):D_b:=\{z\in\C\mid |z|^2<2z_b\} \to H^1_r(\R^3)\times\R,}
such that $(\fy,\om)=(\Phi[z],\Om[z])$ is a soliton for each $z\in D_b$, with a decomposition 
\EQ{
 \Phi[z]=z\phi_0+\ga[z], \pq \Om[z]=-e_0+o(z),\pq \ga[z]\perp\phi_0, \pq \|\ga[z]\|_{H^1}=o(|z|^2),}
satisfying the gauge covariance $(\Phi[e^{i\te}z],\Om[e^{i\te}z])=(e^{i\te}\Phi[z],\Om[z])$, and $\bM(\Phi[z])=|z|^2/2+o(|z|^4)$ is an increasing function of $|z|$. 
Moreover, there is an open set in $H^1_r(\R^3)$ which contains $\Phi[D_b]$ but no other soliton. 
\end{lem}
Let $\mu_b>0$ be the maximal mass among those solitons:
\EQ{ \label{def mub}
 \mu_b := \sup \bM(\Phi[D_b]).}
Then the monotonicity implies that 
\EQ{
 [0,z_b)\ni z \mapsto \bM(\Phi[z])\in [0,\mu_b)}
is an increasing bijection. Let $z_0:[0,\mu_b)\to[0,z_b)$ be the inverse function, so that 
\EQ{
 \bM(\Phi[z_0(\mu)])=\mu.} 

The following lemma shows that the above solitons are the ground states for small mass. It will be crucial also for identifying the first excited state. 
\begin{lem}[Small mass dichotomy] \label{lem:SMD}
For any $\fy\in H^1(\R^3)$ satisfying $\bK_2(\fy)\ll \bM(\fy)^{-1}$ and $\bM(\fy)\ll 1$, we have one of the following {\rm (i)-(iii)}
\begin{enumerate}
\item $\bH^0(\fy) \lec \bM(\fy)$, 
\item $\bM(\fy) \lec \bH^0(\fy) \sim \bE(\fy) \sim \bK_2(\fy)$,
\item $\sg=+$ and $\bG(\fy) \gec \bH^0(\fy)\gec \bM(\fy)^{-1}$, \label{large case}
\end{enumerate}
Moreover, in the case {\rm(\ref{large case})}, we have 
\EQ{ \label{G dom V}
 |\ml{V}(\fy)|+|\ml{\cS'_pV}(\fy)|+|\ml{\cS'_p\cS'_qV}(\fy)| \ll \bH^0(\fy)} 
for any $p,q>0$. In particular, $\bG$ in {\rm(\ref{large case})} can be replaced with $2\bI$. 
\end{lem}
Note that the first two regions overlap each other, but the last one is separated. Thus the above lemma gives a dichotomy into (i)-(ii) and (iii). The case (ii) can be removed if $\bK_2(\fy)\ll\bM(\fy)$, which is mostly satisfied when the above lemma is used. 

\begin{proof}
Let $\mu:=\bM(\fy)$ and $h:=\bH^0(\fy)$. Using Gagliardo-Nirenberg, we have 
\EQ{ \label{h-K2}
 \pt |\bG(\fy)| \lec \|\fy\|_4^4 \lec h^{3/2}\mu^{1/2},
 \pr |\ml{W}(\fy)| \lec \|W\|_{L^2+L^\I}(\mu + h^{3/4}\mu^{1/4})\pq(W=V,\cS_p'V,\cS_p'\cS_q'V).}
Splitting into three cases: 
(1) $h\lec \mu$, 
(2) $\mu\ll h\ll\mu^{-1}$,
and (3) $\mu^{-1}\lec h$, 
we may first dispose of (1)=(i). 
In the case of (2), the above estimates imply 
\EQ{
 |2h-\bK_2(\fy)|=|\ml{\cS'_\I V}(\fy)-3\bG(\fy)| \lec \mu+h^{3/4}\mu^{1/4}+h^{3/2}\mu^{1/2} \ll h,}
and the same estimate works for $|\bH^0(\fy)-\bE(\fy)|=|\ml{V}(\fy)-\bG(\fy)|$, leading to (ii). 
In the case of (3), we have $h\gec\mu^{-1}\gg 1$. 
Then using $\bK_2(\fy)\ll\mu^{-1}\lec h$ and $\mu+h^{3/4}\mu^{1/4} \ll h$ in \eqref{h-K2}, we obtain \eqref{G dom V} and 
\EQ{
 h \lec 2h-\bK_2(\fy)-\ml{\cS'_\I V}(\fy)=3\bG(\fy),}
leading to (iii). 
\end{proof}
The above lemma enables us to identify the solitons in Lemma \ref{lem:Phi} with $\Soli_0$: 
\begin{prop}
There exists $0<\mu_d\le\mu_b$ and $0<z_d\le z_b$ such that 
\EQ{ \label{small ground}
 \pt \{\fy\in\Soli_0\mid \bM(\fy)<\mu_d\}=\{\Phi[z]\mid |z|<z_d\},
 \pr 0<\mu<\mu_d \implies \sE_0(\mu)=\bE(\Phi[z_0(\mu)])\sim e_0\mu<0.}
\end{prop}
\begin{proof}
Since $\bK_2=0$ on $\Soli$, we can apply the above lemma to any $\fy\in\Soli$ with $\bM(\fy)<\mu_d\ll 1$, leading to either (i) or (iii). 
Taking $\mu_d>0$ small ensures that the region (i) is in the uniqueness region of $\Phi[D_b]$ in Lemma \ref{lem:Phi}, as well as that the region (iii) is far away. 
Then every $\fy\in\Soli$ with $\bM(\fy)<\mu_d$ is either in $\Phi[D_b]$ or in the region (iii). 
In the latter case, $\fy$ is an excited state, as $\Phi$ gives a soliton with the same mass and negative energy. 
Thus we obtain the first identity in \eqref{small ground}. 
The second one is its obvious consequence. The behavior of $\sE_0$ follows from \eqref{E-M diff} together with the differentiability of $\Phi$ from Lemma \ref{lem:Phi}.
\end{proof}

\subsection{Focusing case $\sg=+$}
Next we investigate the first excited state of small mass in the focusing case. The small mass dichotomy Lemma \ref{lem:SMD} allows us to ignore the potential effect, leading to the same variational characterization as $V=0$ in the higher energy region. 
\begin{prop} \label{Soli foc}
Let $\sg=+$. There exists $0<\mu_e\le\mu_d$ such that for $0<\mu<\mu_e$
\EQ{ \label{char E1}
 \sE_1(\mu)\pt=\inf\{\bE(\fy) \mid \fy\in H^1_r,\ \bM(\fy)=\mu,\ \bK_2(\fy)=0,\ \bG(\fy)\ge 1\}
 \pr=\inf\{\bI(\fy) \mid \fy\in H^1_r,\ \bM(\fy)\le\mu,\ \bK_2(\fy)\le 0,\ \bG(\fy)\ge 1\}
 \pr= \mu^{-1}\bM(Q)\bE^0(Q)(1+o(1)) \gg 1 \pq (\mu\to 0),}
where $\inf$ is attained by some $\fy\in\Soli_1$, and 
$\sE_1(\mu)$ is decreasing in $\mu$. 
Moreover, there is a continuous function $\ka:(0,\mu_e)\times(0,\I)\to(0,1/2)$ such that for any $\de>0$ and any $\fy\in H^1_r$ satisfying $\bM(\fy)<\mu_e$, $\bE(\fy)\le \sE_1(\mu)-\de$ and $\|\na\fy\|_2\ge 1$, we have 
\EQ{ \label{bd K2}
 |\bK_2(\fy)| \ge \ka(\bM(\fy),\de).}
\end{prop}
The above minimization is well known in the case $V=0$ without the restriction on $\bG$. 
Some restriction to higher energy is needed in the case $V\not=0$, since those $\inf$ in \eqref{char E1} would become $\sE_0(\mu)$ without $\bG\ge 1$. 
\begin{proof}
First, the $\ge$ part of \eqref{char E1} is obvious from $\bK_2=0$ on $\Soli$, the dichotomy Lemma \ref{lem:SMD}, and $\bI=\bE-\bK_2/2$. 
The second infimum is obviously decreasing in $\mu$. 

To show the $\le$ part of the second equality, let $\fy\in H^1_r$, $\bM(\fy)<\mu$, $\bK_2(\fy)\le 0$ and $\bG(\fy)\ge 1$. 
Consider the one-parameter scaling $v(t):=\cS_{3.5}^t\fy$ for $t\le 0$. 
The dichotomy implies $\bG(\fy)\gec\mu^{-1}\gg 1$. 
Since $\cS'_{3.5}\bM=-9\bM/7<0$, there exists $T<0$ such that $\bM(v(T))=\mu$. Moreover, since 
\EQ{
 \cS_{3.5}'\bI=\frac{3}{14}\bG-\frac{1}{2}\ml{\cS_{7/3}'\cS_{3/2}'V},
 \pq \cS_{3.5}'\bK_2=\frac{5}{7}\bK_2+\frac{6}{7}\bG+\ml{\cS'_{3/2}V},}
we have, using Lemma \ref{decop V}, 
\EQ{ \label{I' K2' bd}
 \cS_{3.5}'\bI(v)\gec \mu^{-1}, \pq \cS_{3.5}'\bK_2(v)>5(\bK_2(v)+1)/7} 
as long as $\bK_2(v)\le\bM(v)\le \mu$. 
The second inequality of \eqref{I' K2' bd} implies that if $\bK_2(v)\ge -1$, $\bK_2(v)$ is decreasing as $t$ decreases. 
Therefore $\bK_2(v)<0$ and \eqref{I' K2' bd} are preserved for $0>t>T$, 
hence the infimum is reduced to the case $\bM(\fy)=\mu$. 

Next consider the one-parameter scaling $v(t):=\cS_2^t\fy$ for $t\le 0$. 
Since 
\EQ{ \label{S2'I}
 \cS'_2\bI=3\bG/2-\ml{\cS'_\I\cS'_{3/2}V}/2, 
 \pq \cS'_2\bK_2=2\bK_2-2\cS'_2\bI,} 
a similar argument as above implies that $\bI(v)$ is decreasing and $\bK_2(v)$ is increasing as $t$ decreases, as long as $\bK_2(v)<0$, while $\bG(v)\ge 1$ is preserved by the dichotomy. 
Thus the infimum is further reduced to the case $\bK_2(\fy)=0$, which means the second equality in \eqref{char E1}. 

To prove the existence of minimizer as well as the lower bound \eqref{bd K2} on $|\bK_2|$, take any sequence $\fy_n\in H^1_r$ satisfying $\bM(\fy_n)\to\mu\in(0,\mu_e)$, $\bK_2(\fy_n)\to 0$, $\bG(\fy_n)+\|\na \fy_n\|_2\ge 1$ and $\bE(\fy_n)\to E_\I\le E_*(\mu)$, where $E_*$ is the infimum in \eqref{char E1}. 
Using Lemma \ref{decop V} with Gagliardo-Nirenberg, we have 
\EQ{ \label{bd by K2}
 \bH^0=3\bE-\bK_2-\ml{\cS'_1V} \le 3\bE-\bK_2+C\bM+\bH^0/2.}
Hence $\fy_n$ is bounded in $H^1_r$, and so, we may assume, passing to a subsequence, that $\fy_n\to\exists\fy\in H^1_r$ weakly. 
Since $\bK_2=2\bH^0-3\bG-\ml{\cS'_\I V}$, disposing of the potential part as above, we deduce that $\bG(\fy_n)\gec 1$ and $\|\na\fy_n\|_2\gec 1$ are equivalent for large $n$, and then the dichotomy implies $\bH^0(\fy_n)\sim \bG(\fy_n)\gec \mu^{-1}$. 

Since $\bG$ and the potential functionals are weakly continuous on $H^1_r$, we have 
\EQ{
 \bI(\fy)=E_\I, \pq \bG(\fy) \ge 1, \pq \bM(\fy)\le\mu, \pq \bK_2(\fy)\le 0,}
hence $E_\I=E_*$ and $\fy$ is a minimizer of \eqref{char E1}. 
Moreover, the above argument implies that $\bM(\fy)=\mu$ and $\bK_2(\fy)=0$. 
Since the dichotomy implies $\bG(\fy)\gec\mu^{-1}\gg 1$, we have Lagrange multipliers $\om,\al\in\R$ such that 
\EQ{ 
 \bE'(\fy) + \om \bM'(\fy) = \al \bK_2(\fy).}
Differentiation along the curve $\cS_2^t\fy$ at $t=0$ yields
\EQ{
 0 = \bK_2(\fy) = \cS'_2\bE(\fy) = \al \cS'_2\bK_2(\fy)=-2\al\cS_2'\bI(\fy),}
where $\cS'_2\bI(\fy)\not=0$ by \eqref{S2'I} with the dichotomy, 
hence $\al=0$. 
This means that $\fy\in\Soli$ and so $\sE_1(\mu)=E_*(\mu)$, as well as the lower bound \eqref{bd K2} on $|\bK_2|$. 

Finally, we prove the asymptotic formula. 
In the second infimum in \eqref{char E1}, put $\psi(x):=\mu\fy(\mu x)$ and $V_\mu(x):=\mu^2V(\mu x)$. Then 
\EQ{ \label{muE1 rescaled}
 \mu \sE_1(\mu)\pt= \min_{\psi\in A_\mu}\left\{\bG(\psi)/2+\ml{(\cS'_{3/2}V)_\mu}(\psi)\right\},
 \pr A_\mu:=\{\psi\in H^1_r \mid \bM(\psi)\le 1,\ \bK_2^0(\psi)\le\ml{(\cS'_\I V)_\mu}(\psi),\ \bG(\psi)\ge \mu\}.}
Since $\|V_\mu\|_{L^2+L^\I}\le \mu^{1/2}\|V\|_{L^2+L^\I}$, we have for $p>0$
\EQ{
 |\ml{(\cS'_pV)_\mu}(\psi)| \lec \mu^{1/2}(\|\psi\|_4^2+\|\psi\|_2^2).}
Hence if $\psi\in H^1_r$ satisfies $\bM(\psi)\le 1$ and $\bK_2^0(\psi)\le -1$, then $\psi\in A_\mu$ for $0<\mu\ll 1$. 
Therefore as $\mu\to 0$, the minimizer $\psi$ is bounded in $L^4_x$. 
Since $\bG(\fy)\gec\mu^{-1}$, we obtain $\bG(\psi)\sim 1$ and $|\ml{(\cS'_pV)_\mu}(\psi)|\le O(\mu^{1/2})$. 
On the other hand, for any $\psi\in H^1_r$ satisfying $\bM(\psi)\le 1\sim\bG(\psi)$ and $\bK_2^0(\psi)=0$, we deduce from \eqref{S2'I} that for $0<\mu\ll 1$ there exists $t=O(\mu^{1/2})$ such that $\cS_2^t\psi\in A_\mu$, using the implicit function theorem around $t=0$.  
Therefore
\EQ{ \label{Q mini}
 \lim_{\mu\to 0}\mu \sE_1(\mu)=\inf\{\bG(\psi)/2 \mid \psi\in H^1_r,\ \bM(\psi)\le 1,\ \bK_2^0(\psi)\le 0,\ \psi\not=0\}.}
To see that the above equals $\bM(Q)\bE^0(Q)$, we may first replace $\bK_2^0(\psi)\le 0$ with $\bK_2^0(\psi)=0$, since on the curve $\R\ni t\mapsto \cS_2^t\psi\in H^1_r$ for any $\psi\in H^1_r\setminus\{0\}$, $\bM$ is constant, $\bG$ is increasing, and $\bK_2^0$ changes its sign exactly once and from positive to negative. 
Next, since $\bG(\cS_3^t\psi)=e^t\bG(\psi)$, $\bM(\cS_3^t\psi)=e^{-t}\bM(\psi)$ and $\bK_2^0(\cS_3^t\psi)=e^t\bK_2^0(\psi)$, we may remove $\bM(\psi)\le 1$ by replacing the minimized quantity $\bG/2$ with $\bM\bG/2$, which may further be replaced with $(\bG/2+\bM)^2/4$, because 
\EQ{
 \inf_{t\in\R}(\bG/2+\bM)(\cS_3^t\psi)=\inf_{t\in\R}[e^t\bG(\psi)/2+e^{-t}\bM(\psi)]=2\sqrt{\bM(\psi)\bG(\psi)}.}
Since $\bG/2=\bE^0$ on $\bK_2^0=0$, we thus obtain 
\EQ{
 \eqref{Q mini}=\left[\inf\{(\bE^0+\bM)(\psi)\mid \psi\in H^1_r\setminus\{0\},\ \bK_2^0(\psi)=0\}\right/2]^2.}
It is well-known that the above infimum is attained by the ground state $Q$ (see, e.g., \cite[Lemma 2.1]{NSK}). 
Using that $\bK_2^0(Q)=0=\p_{\al=1}(\bE^0+\bM)(\al Q)$, it is elementary to see that the above equals $\bM(Q)\bE^0(Q)$. 
\end{proof}

\subsection{Defocusing case $\sg=-$}
In the defocusing case, the variational structure is much simpler, and so we can determine the entire set of solitons $\Soli$ without the mass constraint. 
In this subsection, we prove the following
\begin{prop} \label{Soli defoc}
Let $\sg=-$. Under the assumptions on $V$ in Section \ref{ss:asm V}, the equation \eqref{sNLSP} has a unique positive solution $\fy_\om$ for each $\om\in(0,-e_0)$, and 
\EQ{ 
 \Soli=\{e^{i\te}\fy_\om \mid 0<\om<-e_0,\ \te\in\R\}.}
The function $(0,-e_0)\ni\om\mapsto\bM(\fy_\om)\in(0,\I)$ is $C^1$, decreasing and bijective. Let $\om_0(\mu)$ be its inverse function. Then for all $\mu>0$,
\EQ{
 e_0\mu < \sE_0(\mu)=\bE(\fy_{\om_0(\mu)}) < 0 < \I=\sE_1(\mu),\pq \om_0'(\mu)<0.}
\end{prop}

Since $H\ge e_0$, multiplying \eqref{sNLSP} with $\fy$ 
\EQ{
 0=\LR{(H+\om)\fy|\fy}+\|\fy\|_4^4 = \|\na\fy\|_2^2+\om\|\fy\|_2^2+\LR{V\fy|\fy}+\|\fy\|_4^4}
implies that $\om<-e_0$ is necessary for existence of a non-trivial solution. 

If $\om\le 0$, then putting $\psi:=r\fy$ we have from \eqref{sNLSP}
\EQ{
 \psi_{rr}+\om\psi=(V+|\fy|^2)\psi,\pq\liminf_{r\to\I}\BR{|\psi(r)|+|\psi_r(r)|}=0.}
Rewriting the above into an integral equation from $r=\I$, we obtain  
\EQ{
 |\psi(s)| \le \int_s^\I(r-s)|(V+|\fy|^2)\psi|dr
 \le \|\psi\|_{L^\I(r>s)}\|r(V+|\fy|^2)\|_{L^1(r>s)},}
where the last norm is vanishing as $s\to\I$ by the assumption $V/|x|\in L^1(\R^3)$ and $\fy \in L^2(\R^3)$. 
Hence taking $s\to\I$ and then solving the ODE, we deduce that $\fy=0$. 

Therefore $0<\om<-e_0$ for every non-trivial $\fy\in\Soli$. 
Moreover, using Lemma \ref{decop V} with $\e=\om/2$, we deduce that $\fy\in\Soli$ is uniformly bounded in $H^1(\R^3)$ on any interval of $\om$ away from $0$, while $\fy_\om\to 0$ in $H^1(\R^3)$ as $\om\to-e_0-0$. 

For each $\om\in(0,-e_0)$, we have a solution $\fy_\om$ of \eqref{sNLSP} which is a global minimizer
\EQ{ \label{global min}
 (\bE+\om\bM)(\fy_\om)=\inf\{(\bE+\om\bM)(\fy)\mid \fy\in H^1\}<0.}
The proof is easy and omitted. The positivity is also standard. The uniqueness of the solution $\fy_\om$ for each $\om$, modulo the phase $e^{i\te}$, follows from a general argument:
\begin{lem}
Let $H$ be a self-adjoint operator on $L^2$ with non-degenerate eigenvalue $e_0$, 
and assume the rest of the spectrum of $H$ is contained 
in $[e_1, \infty)$ for some $e_1 > e_0$. 
Let $f:[0,\I)\to[0,\I)$ be a strictly monotone function such that $f(a)a$ is non-decreasing. Then the nonlinear eigenvalue problem
\[
  (H + f(|\fy|)) \fy = \om \fy
\]
can have at most one non-trivial solution (up to the phase symmetry) for each $\om < e_1$. The same conclusion holds for $\om=e_1$ if $f(a)a$ is strictly increasing. 
\end{lem}
The above lemma may be known, but a proof is given below for the sake of completeness. 
\begin{proof}
Let $f(z):=f(|z|)$ for $z\in\C$. Let $\fy$ and $\psi$ be two non-zero solutions, and let $\phi_0$
be an eigenfunction of $H$ for $e_0$. We must have $(\phi_0|\fy) \not= 0$, 
or else 
\EQ{
  (e_1 - \om) \| \fy \|_2^2 \le 
  (\fy|(H - \om) \fy) = - (|\fy|^2|f(\fy)),}
which contradicts either $\om<e_1$ or $\om=e_1$ with strictly increasing $f(a)a$. 
Thus we can find $\be \in \C\setminus\{0\}$ so that $(\phi_0|\be \fy + \psi) = 0$. 
Using the invariance of the equation for $\fy\mapsto e^{i\te}\fy$, we may take $\be>0$ by appropriate complex rotation of $\fy$. Then 
\EQ{ \label{min comb}
  \pn(e_1-\om) \| \be \fy + \psi \|_2^2 
  \pt\le \LR{\be \fy + \psi| (H-\om)(\be \fy + \psi) }
  \pr= -\be^2\LR{f(\fy)||\fy|^2}-\LR{f(\psi)||\psi|^2} + 2\be\LR{\fy|(H-\om)\psi}
  \pr\le 2\be\left[|\LR{\fy|(H-\om)\psi}|-\sqrt{\LR{f(\fy)||\fy|^2}\LR{f(\psi)||\psi|^2}}\right].}
First consider the case where $f$ is non-decreasing. Then using Schwarz,   
\EQ{
 |\LR{\fy|(H-\om)\psi}| =\CAS{|\LR{\fy|f(\psi)\psi}|\le\sqrt{\LR{f(\psi)||\fy|^2}\LR{f(\psi)||\psi|^2}}, \\
  |\LR{f(\fy)\fy|\psi}|\le\sqrt{\LR{f(\fy)||\fy|^2}\LR{f(\fy)||\psi|^2}}.}}
So we arrive at
\EQ{ 
 \pt \LR{f(\psi)||\fy|^2} \ge \LR{f(\fy)||\fy|^2},
 \pq \LR{f(\fy)||\psi|^2} \ge \LR{f(\psi)||\psi|^2},}
and hence $\LR{f(\fy)-f(\psi)||\fy|^2-|\psi|^2} \le 0$. 
Since $f$ is non-decreasing, this implies that $f(\fy)=f(\psi)$ (a.e.).

Next consider the case where $f$ is non-increasing. By Schwarz, we have 
\EQ{
 \pt|\LR{\fy|f(\psi)\psi}| \le\sqrt{\LR{f(\fy)||\fy|^2}\LR{f(\psi)^2|\psi|^2|1/f(\fy)}},
 \pr|\LR{f(\fy)\fy|\psi}|\le\sqrt{\LR{f(\fy)^2|\fy|^2|1/f(\psi)}\LR{f(\psi)||\psi|^2}}.}
Hence \eqref{min comb} implies that 
\EQ{
 \pt \LR{f(\psi)^2|\psi|^2|1/f(\fy)} \ge \LR{f(\psi)^2|\psi|^2|1/f(\psi)},
 \pr \LR{f(\fy)^2|\fy|^2|1/f(\psi)} \ge \LR{f(\fy)^2|\fy|^2|1/f(\fy)},}
and so,
$\LR{f(\fy)^2|\fy|^2-f(\psi)^2|\psi|^2|1/f(\fy)-1/f(\psi)} \le 0$. 
Since $1/f(a)$ and $f(a)a$ are both non-decreasing, we have $f(\fy)=f(\psi)$ or $f(\fy)|\fy|=f(\psi)|\psi|$ (a.e.). 
If $e_1>\om$, then $\psi=-\be\fy$, otherwise the above must be a strict inequality, contradicting the monotonicity of $f(\fy)$ and $f(\fy)|\fy|$. 
If $e_1=\om$, then $f(z)|z|$ is strictly increasing, so we get $f(\fy)=f(\psi)$, and then going back to \eqref{min comb}, 
\EQ{
  0= (e_1-\om) \|\be \fy + \psi\|_2^2 
  \pt\leq \LR{\be \fy + \psi|(H-\om)(\be \fy + \psi)} 
  \pr=  - \LR{f(\fy)||\be \fy + \psi|^2}\le 0. }
The strict monotonicity of $f(a)a$ implies that at each $x$, $f(\fy(x))=0$ implies $f(\psi(x))=0$ and $\fy(x)=0=\psi(x)$. Hence $\psi=-\be\fy$ (a.e.). 

Thus we obtain $\psi=-\be\fy$ anyway. Then the equation for $\fy$ and $\psi$ implies that 
\EQ{
(\om-H)\fy=f(\fy)\fy=-f(\psi)\psi/\be=f(\be\fy)\fy.} 
Since $f$ is strictly monotone, this implies $\be=1$ or $\fy=0$ a.e.
\end{proof}
Once we have the uniqueness of $\fy_\om$ for $\om$, it is easy to prove continuity and then differentiability in $\om$. Differentiating the equation 
\EQ{
 (H+\om+3\fy_\om^2)\fy_\om'+\fy_\om=0,\pq \fy_\om':=\p_\om\fy_\om}
and multiplying it with $\fy_\om'$ yield
\EQ{ 
 \p_\om\bM(\fy_\om)=\LR{\fy_\om|\fy_\om'}=-\LR{(H+\om+3\fy_\om^2)\fy_\om'|\fy_\om'} \le -2\|\fy_\om\fy_\om'\|_2^2<0,}
where we used $H+\om+\fy_\om^2\ge 0$, because $\fy_\om>0$ is the ground state in the kernel of this Schr\"odinger operator. 
Hence $\bM(\fy_\om)$ is decreasing in $\om$. 
Moreover, $\bM(\fy_\om)\to\I$ as $\om\to+0$, since otherwise the weak limit yields $0\not=\fy_0\in H^1_r$ satisfying $H\fy_0+\fy_0^3=0$, which is impossible. 
Using \eqref{E-M diff} as well, we conclude the proof of Proposition \ref{Soli defoc}.   

\section{Blow-up below the excited energy} \label{ss:bup}
We are now ready to prove the blow-up part of Theorem \ref{main}, using the above characterization of $\Soli_1$ together with the estimate on $\bK_2$, namely Proposition \ref{Soli foc}.

Let $u$ be a solution of \eqref{NLSP} with $\sg=+$, satisfying \eqref{bup cond} as well as $\bM(u)=:\mu<\mu_e$ and $\bE(u)<\sE_1(\mu)$, where $\mu_e$ is the small mass condition of Proposition \ref{Soli foc}. 
Fix $\de>0$ such that $\bE(u)\le \sE_1(\mu)-\de$. Suppose for contradiction that $u$ exists on $0<t<\I$. 
Then Proposition \ref{Soli foc} and Lemma \ref{lem:SMD} together with the continuity of $u(t)$ in $H^1_x$ imply that \eqref{bup cond} is preserved for all $t>0$, and also from \eqref{bd K2} 
\EQ{
 \bK_2(u(t)) \le -\ka(\mu,\de).} 

We have the saturated virial identity from \cite{OT}
\EQ{
 \p_t\LR{R f_Ru|iu_r} \pt= 2\bK_2(u) - \int[2|u_r|^2f_{0,R}+R^{-2}|u|^2f_{1,R}+|u|^4f_{2,R}]dx
 \prQ-\int(1-f_RR/r)|u|^2 x\cdot\na V dx,}
where $f_R(x)=f(x/R)$ with $R\gg 1$, and $f_{j,R}(x)=f_j(x/R)$ are derived from $f$ by 
\EQ{
 f_0=1-f_r\ge 0, \pq f_1=\De(\p_r/2+1/r)f, \pq f_2=-3/2+(\p_r/2+1/r)f,}
while $f(x)=f(|x|)$ is chosen to be smooth radial satisfying 
\EQ{
 f(r)=\CAS{r &(r\le 1),\\ 3/2 &(r\ge 2).}}
The $|u|^4$ integral is bounded by the radial Sobolev inequality 
\EQ{
 \|u\|_{L^4(|x|>R)}^4 \lec R^{-2}\|u\|_{L^2(|x|>R)}^3\|u_r\|_{L^2(|x|>R)}.}
Then we obtain 
\EQ{
 \pt-\int[2|u_r|^2f_{0,R}+R^{-2}|u|^2f_{1,R}+|u|^4f_{2,R}+(1-f_RR/r)|u|^2 x\cdot\na V]dx 
 \pr\lec R^{-4}\|u\|_{L^2(|x|>R)}^6 + o(1)\|u\|_{L^2(|x|>R)}^2,}
as $R\to\I$. 
See \cite{OT}, \cite[\S 4.1]{NSS}, for the detail. 
Note that the potential part was treated by \eqref{V decay}, using $1-f_RR/r=0$ for $|x|<R$. Hence, for large $R$, we have 
\EQ{
 \p_t\LR{R f_Ru|iu_r} < -\ka(\mu,\de)<0.}
Since $\|u\|_{L^2_x}$ is conserved, it implies that $\|u_r\|_2\to \I$ as $t\to\I$. Then as $t\to\I$, 
\EQ{
 \bK_2(u) \pt=3\bE(u)-\bH^0(u)-\ml{\cS'_1 V}(u) \sim -\|\na u\|_2^2.}
The rest of the proof is the same as in the case without the potential, see \cite{OT}. 
Thus we obtain the ``if"-part for \eqref{bup cond} of the blow-up in Theorem \ref{main}. 

Next we show the ``only if"-part of \eqref{bup cond}, namely the global existence when it is not satisfied. 
If $\sg=-$, then we have a priori $H^1$ bound by conservation of the energy and mass, disposing of the potential part by Lemma \ref{decop V}, which leads to the global well-posedness in $H^1(\R^3)$. 

Hence we may restrict to the case $\sg=+$, $\bM(u)=\mu<\mu_e$ and $\bE(u)<\sE_1(\mu)$. 
By the persistence of \eqref{bup cond} proved above, if \eqref{bup cond} is initially not satisfied, neither is it at any other time. 
If $\bK_2(u(t))<0$ and $\|\na u(t)\|_2\le 1$, then Lemma \ref{lem:SMD} implies that $\|u(t)\|_{H^1}^2\lec\mu\ll 1$. 
If $\bK_2(u(t))\ge 0$, then \eqref{bd by K2} yields a priori bound on $H^1_x$ by the mass-energy conservation. 
Hence the solution $u$ is global and bounded in $H^1_x$ for all $t\in\R$. 
Moreover, we have the scattering to the ground states by \cite{gnt} if $\bH^0(u(t))\lec\mu$, and it is preserved for all $t\in\R$. 

Thus we have obtain the global existence part of Theorem \ref{main}. 
\begin{lem} \label{lem:gs}
For any $u(0)\in H^1_r$ satisfying $\bM(u(0))=\mu\le \mu_e$ and $\bE(u(0))=\e<\sE_1(\mu)$, 
the corresponding solution $u$ of \eqref{NLSP} is global in time iff \eqref{bup cond} fails. 
Then it is never satisfied at any $t\in\R$. 
Moreover, the global solution $u$ satisfies one of the following
\begin{enumerate}
\item $\|u(t)\|_{H^1_x}^2\lec\mu$ for all $t\in\R$, scattering to $\Soli_0$. 
\item $\mu \lec \|u(t)\|_{H^1_x}^2 \lec \e+\mu$ and $\bK_2(u(t))\ge \ka(\mu,\sE_1(\mu)-\e)$ for all $t\in\R$,
\end{enumerate}
where $\ka>0$ is the same as in \eqref{bd K2}. 
\end{lem}
The rest of this paper is devoted to the scattering in (ii). 

\section{Modulation and linearized equations around the ground state}
Here we recall the coordinate in \cite{gnt} around the small ground state, and observe that it can be applied to large solutions as long as the mass $\bM(u)$ is small, including the excited solitons. 
For any $\mu>0$, denote 
\EQ{ \label{def H1mu}
 H^1[\mu]:=\{\fy \in H^1_r(\R^3) \mid \bM(\fy)< \mu\}.}
Let $\Phi:D_b\to H^1_r$ be the small ground states as in Lemma \ref{lem:Phi}. 
We have the following nonlinear projection to them. 
\begin{lem} \label{lem:decop2Phi}
There exist $0<\mu_p<\mu_b$ and a unique mapping $H^1[\mu_p]\ni\fy\mapsto (z,\y)\in D_p\times H^1[\mu_p]$, where $D_p:=\{z\in\C\mid |z|^2<2\mu_p\}$, such that 
\EQ{ \label{def decop}
 \pt \fy=\Phi[z]+\y,
 \pr \y\in \cH_c[z]:=\{\psi\in H^1(\R^3)\mid \LR{i\psi|\p_j\Phi[z]}=0\ (j=1,2)\},}
where $\p_j$ denotes the derivative with respect to the real and imaginary parts of $z=z_1+iz_2$. 
Moreover, the map $\fy\mapsto(z,\y)$ is smooth and injective from $H^1[\mu_p]$ to $\C\times H^1_r$. Furthermore, the orthogonal projection $P_c$ to the continuous spectrum subspace 
\EQ{ \label{def Pc}
 P_c\fy:=1-\phi_0(\fy|\phi_0)} 
is bijective from $\cH_c[z]$ onto 
\EQ{
 \cH_c[0]=P_c H^1(\R^3)=\{\fy\in H^1(\R^3) \mid \LR{i\fy|\phi_0}=0=\LR{i\fy|i\phi_0}\},}
for any $z\in D_p$, and 
\EQ{ \label{def Rz}
 R[z]:=(P_c|_{\cH_c[z]})^{-1}} 
is a compact and continuous perturbation of identity in the operator norm on any space between $H^2\cap W^{1,1}$ and $H^{-2}+L^\I$. 
\end{lem}
In \cite{gnt}, the above coordinate was defined on a small ball of $H^1(\R^3)$. However, it is easy to see that the smallness in $L^2$ suffices, since $z$ is determined by solving the orthogonality conditions 
\EQ{
 \LR{\fy-\Phi[z]|i\p_j\Phi[z]}=0 \pq (j=1,2),}
by the implicit function theorem. The derivative of the left hand side equals
\EQ{
 \LR{\fy-\Phi[z]|i\p_k\p_j\Phi[z]}-\LR{\p_k\Phi[z]|i\p_j\Phi[z]},\pq (j,k\in\{1,2\}).}
The second term is a non-degenerate $2\times 2$ matrix of $O(1)$, while the first term is bounded by $\|\fy-\Phi[z]\|_2\lec\sqrt{\mu_p}\ll 1$, thereby the implicit function theorem works, leading to the conclusion. The $L^2$ bound follows from the orthogonality
\EQ{
 \bM(\fy)=\bM(\Phi[z])+\bM(\y).}
The operator $R[z]$ is linear, so it does not need any smallness condition. Actually, the above lemma holds without even assuming that the function is in $H^1$. 
Hence every solution $u$ in $H^1[\mu_p]$ can be written as 
\EQ{
 u(t) = \Phi[z(t)] + R[z(t)]\x(t) = \Phi[z(t)]+\y(t),}
uniquely, and the equation for $u$ can be rewritten for $(z,\x)$ as, regarding $\C\simeq\R^2$, 
\EQ{ \label{eq zxi}
 \CAS{ \dot z-i\Om[z]z =\U{N}(z,R[z]\x):=M(z,R[z]\x)^{-1}\LR{N(z,R[z]\x)|D\Phi[z]},
 \\ i\dot\x+H\x = B[z]\x+\ti N(z,\x),} }
where $M(\cdot)$ is a $2\times 2$ matrix and $N(\cdot,\cdot)$ is a scalar defined by 
\EQ{ \label{def MN}
 \pt M_{j,k}(z,\y):=\LR{i\p_j\Phi[z]|\p_k\Phi[z]}-\LR{i\y|\p_j\p_k\Phi[z]},
 \pr N(z,\y):=\sg\BR{2\Phi[z]|\y|^2+\overline{\Phi[z]}\y^2+|\y|^2\y},}
$B[z]$ is the ``potential part" by the small soliton, namely
\EQ{ \label{def Bz}
 B[z]\x:=\sg P_c\{2|\Phi[z]|^2R[z]\x+\Phi[z]^2\ba{R[z]\x}\},}
which is $\R$-linear but not $\C$-linear, and $\ti N(\cdot,\cdot)$ is the quadratic part 
\EQ{ \label{def tiN}
 \ti N(z,\x):=P_c\{N(z,R[z]\x)-iD\Phi[z]\U{N}(z,R[z]\x)\}.}

We introduce some notation for the linearized solutions. For any $s\in\R$ and any set $X$, the set of $X$-valued functions defined around $s$ is denoted by 
\EQ{  \label{def germ}
 X\{s\}:=\{u:I\to X \mid s\in \exists I\subset\R\}.}
For any interval $I\subset\R$, $z\in C(I;\C)$, $s_0\in I$ and $u\in H^1\{s_0\}$, the linear solution $v$ of 
\EQ{ \label{Leqxi}
 i\dot v + Hv = B[z]v, \pq v(s_0)=u(s_0)}
is denoted by 
\EQ{ \label{def propa}
 u[z,s_0]:=v \in C(I;H^1).}
Note that this depends on $u(s_0)$ but not on $u(t)$ at the other time $t\not=s_0$. Indeed, there is no point for $u$ to depend on $t$ in the definition, but this convention avoids writing the same time $s_0$ twice. We can apply it to time-independent $u$ as well.  Obviously
\EQ{
 \forall s_1\in I,\pq u[z,s_0][z,s_1]=u[z,s_0],}
while the solution without the potential $B[z]$ is given by 
\EQ{
 u[0,s_0]=e^{i(t-s_0)H}u(s_0).}
The associated Duhamel integral is denoted by 
\EQ{ \label{def Duh}
 \D f[z,s_0](t):=\int_{s_0}^t f[z,s](t)ds,}
so that $v:=\D f[z,s_0]$ satisfies 
\EQ{
 i\dot v+Hv = B[z]v+ f, \pq v(s_0)=0.}
Hence for any $\fy\in H^1$, $z\in\C\{s_0\}$ and $f\in H^1\{s_0\}$, the solution of 
\EQ{
 i\dot\x+H\x=B[z]\x+ f,\pq \x(s_0)=\fy}
is uniquely given by
\EQ{
 \x=\fy[z,s_0]+\D f[z,s_0].}

Another notation 
\EQ{ \label{def turnoff}
 u[z,s_0]_> :=\CAS{ u(t) &(t<s_0) \\ u[z,s_0] &(t>s_0)}}
is convenient to ``turn off the nonlinearity" after some time. 
Indeed if $u$ solves 
\EQ{ \label{Lequf}
 i\dot u + Hu = B[z]u + f}
and $s_0<s_1$, then we have 
\EQ{
 \CAS{ u = u[z,s_0] + \D f[z,s_0], \\ u[z,s_1]_> = u[z,s_0] + \D 1_{t<s_1} f[z,s_0].}}

Next a few (semi)norms are introduced for space-time functions. For $s\in\R$, put
\EQ{ \label{def Stz}
 \Stz^s := L^\I_t H^s_x \cap L^2_t B^s_{6,2}, 
 \pq \Stz^{*s} := L^1_t H^s_x + L^2_t B^s_{6/5,2}, 
 \pq \ST:= L^4_t L^6_x.}
$\Stz^s$ is the full Strichartz norm for $H^s$ solutions, 
and $\Stz^{1/2} \subset L^4_t B^{1/2}_{3,2}\subset \ST$ by interpolation and Sobolev. 

The next semi-norm is a bit more involved. It is needed for long-time perturbation argument for the radiation part $\x$, whose equation contains quadratic terms. For $-\I\le T_0<T_1\le T_2\le\I$, $z\in C((T_0,T_2);\C)$ and $u\in C((T_0,T_1);H^1)$, 
\EQ{ \label{def wavenorm}
 \|u\|_{[z;T_0,T_1;T_2]}:=\sup_{T_0<S<T<T_1}\|u[z,T]_>-u[z,S]_>\|_{\ST(S,T_2)}}
is a semi-norm vanishing exactly for solutions of the linearized equation with the parameter $z$, namely 
\EQ{ \label{ker [z]}
 \|u\|_{[z;T_0,T_1;T_2]}=0 \iff i\dot u+Hu=B[z]u \pq \text{on $(T_0,T_1)$.}}
If $T_0>-\I$, we can fix $S\to T_0$ to get an equivalent semi-norm 
\EQ{
 \|u\|_{[z;T_0,T_1;T_2]'}\pt:=\sup_{T_0<T<T_1}\|u[z,T]_>-u[z,T_0]_>\|_{\ST(T_0,T_2)} 
 \pr\le \|u\|_{[z;T_0,T_1;T_2]} \le 2\|u\|_{[z;T_0,T_1;T_2]'},}
where the first inequality follows from the continuity as $S\to T_0+0$, while the second one is obvious by the triangle inequality. 

This semi-norm measures how much $u$ deviates from the linear evolution between $T_0$ and $T_1$ and its influence until $t<T_2$. If $u$ solves \eqref{Lequf} on $(T_0,T_1)$, then 
\EQ{
 \pt \|u\|_{[z;T_0,T_1;T_2]}=\sup_{T_0<S<T<T_1}\|\D 1_{t<T}f[z,S]\|_{\ST(S,T_2)},
 \pr \|u\|_{[z;T_0,T_1;T_2]'}=\sup_{T_0<T<T_1}\|\D 1_{t<T}f[z,T_0]\|_{\ST(T_0,T_2)}.} 
Since we use only the Strichartz type estimates, i.e.~$L^p_t$ norms, the right hand side will be estimated in the same way as $\|\D f[z,T_0]\|_{\ST(T_0,T_1)}$. 
It will be used mostly to bound (for $T_0,T_1\in\R$)
\EQ{ \label{use wave norm}
 \|u\|_{[z;T_0,T_1;T_2]} \ge \max(\|u-u[z,T_0]\|_{\ST(T_0,T_1)},\|u[z,T_1]-u[z,T_0]\|_{\ST(T_1,T_2)}).}
The idea of long-time perturbation in this type of norms, together with the use of non-admissible Strichartz (as in Lemma \ref{lem:nonad} below), was introduced in \cite{CNLKG} to treat quadratic and sub-quadratic nonlinearity. 

An advantage of \eqref{def wavenorm} compared with the equivalent form is the monotonicity:
\EQ{ \label{[z] mono}
 T_0\le T_0'<T_1'\le T_1,\ T_2'\le T_2 \implies \|u\|_{[z;T_0',T_1';T_2']}
 \le \|u\|_{[z;T_0,T_1;T_2]},}
which is obvious by the definition. 
It is also subadditive for gluing intervals. 
\begin{lem} \label{lem:subadd}
For $T_0<T_1<T_2<T_3$, $z\in C((T_0,T_3);\C)$ and $u\in C((T_0,T_3);H^1)$, 
\EQ{
 \|u\|_{[z;T_0,T_2;T_3]} \le \|u\|_{[z;T_0,T_1;T_3]}+\|u\|_{[z;T_1,T_2;T_3]}.}
\end{lem}
The subadditivity holds also for the equivalent form, which is left for the reader. 
\begin{proof}
The left side is the supremum of $\|u[z,T]_>-u[z,S]_>\|_{\ST(S,T_3)}$ over $T_0<S<T<T_2$. If $T_0<S<T<T_1$ or $T_1<S<T<T_3$, then it is trivially bounded by the first or the second term on the right. 
If $T_0<S\le T_1\le T<T_3$, then we have 
\EQ{
 \|u[z,T]_>-u[z,S]_>\|_{\ST(S,T_3)}
 \pt\le \|u[z,T]_>-u[z,T_1]_>\|_{\ST(S,T_3)}
 \prQ+\|u[z,T_1]_>-u[z,S]_>\|_{\ST(S,T_3)}
 \pr\le \|u\|_{[z;T_0,T_1;T_3]}+\|u\|_{[z;T_1,T_2;T_3]},}
using the continuity of $u[z,R]_>$ in $\ST(S,T_3)$ as $R\to T_1$. 
\end{proof}

Now we derive the standard Strichartz estimates for the linearized equation, with uniformly small $z$. 
\begin{lem} \label{lem:Stz}
Let $I=(T_0,T_1)$ be an interval and $z\in C(I;\C)$ with $\|z\|_{L^\I(I)}\ll 1$.  
Then for any $s_0\in I$, $\fy\in P_cH^1(\R^3)$, $f\in C(I;\cS')$, $T\in(T_0,T_1)$ and $\te\in[0,1]$, 
\EQ{ \label{lin Stz est}
 \pt \|\fy[z,s_0]\|_{\Stz^\te(I)}
  \lec \|\fy(s_0)\|_{  H^\te} \sim \inf_{t\in I}\|\fy[z,s_0](t)\|_{H^\te} \sim \sup_{t\in I}\|\fy[z,s_0](t)\|_{  H^\te},
 \pr \|\D P_cf[z,s_0]\|_{\Stz^\te(I)}
  \lec \|f\|_{\Stz^{*\te}(I)},
 \pq \|\D P_cf[z,s_0]\|_{[z;T_0,T;T_1]} \lec \|f\|_{\Stz^{*1/2}(T_0,T)}.}
Moreover, if $\sup I=\I$ then we have the scattering of $u:=\fy[z,s_0]+\D P_cf[z,s_0]$ as $t\to\I$, namely the strong convergence of $e^{-itH}u(t)$ in $H^1$, and furthermore, $u(t)\to 0$ in $L^6_x$. The same holds for $t\to-\I$. 
\end{lem}
\begin{proof}
Let $\x:=\fy[z,s_0]+\D P_cf[z,s_0]$. Then the Strichartz estimate for $e^{itH}P_c$ yields 
\EQ{
 \|\x\|_{\Stz^\te(I)}
 \pt\lec \|\fy(s_0)\|_{  H^\te}+\|B[z]\x+P_cf\|_{\Stz^{*\te}(I)},}
where the term $B[z]\x$ is bounded by 
\EQ{
 \|B[z]\x\|_{L^2_t  B^\te_{6/5,2}(I)} \lec \|z\|_{L^\I_t(I)}^2\|\x\|_{L^2_t  B^\te_{6,2}(I)} \ll \|\x\|_{L^2_t  B^\te_{6,2}(I)},}
and absorbed by the left. 
This yields the three inequalities in \eqref{lin Stz est}.

If $\sup I=\I$, for any increasing sequence $t_n\to\I$, we have 
\EQ{
 [e^{-itH}\x(t)]_{t=t_m}^{t=t_n}=\int_{t_m}^{t_n}e^{-isH}(B[z]\x+P_cf)ds,}
and using the Strichartz as above, 
\EQ{
 \|[e^{-itH}\x(t)]_{t=t_m}^{t=t_n}\|_{H^1_x} \lec \|\x\|_{L^2_tB^1_{6,2}(t_m,t_n)}+\|f\|_{Stz^{*1}(t_m,t_n)},\pq(t_m<t_n)}
where the right side tends to $0$ as $t_m\to\I$, since the norms consist of $L^p_t$ with $p<\I$. 
Hence $e^{-itH}\x(t)$ converges as $t\to\I$ strongly in $H^1$. Then the decay in $L^6_x$ follows from the $L^{6/5}\to L^6$ decay estimate, the Sobolev embedding $H^1_x\subset L^6_x$ and the density argument. 

For the norm equivalence, let $u:=\fy[z,s_0]$ and $u_0:=\fy[0,s_0]$. Then by the same Strichartz as above, 
\EQ{
 \|u-u_0\|_{\Stz^\te(I)}
 \lec \|B[z]u\|_{L^2_t  B^\te_{6/5,2}(I)} \ll \|u\|_{L^2_t  B^\te_{6,2}(I)},}
and so the right hand side is equivalent to (since $\fy\in P_cH^1$) 
\EQ{
 \|u_0\|_{L^2_t  B^\te_{6,2}(I)} \lec \|u_0(s_0)\|_{  H^\te} \sim \|\LR{HP_c}^{\te/2}u_0(s_0)\|_{L^2} \sim \|u_0\|_{L^\I_t  H^\te_x(I)},}
which implies the first estimate in the lemma. 
\end{proof}
As an immediate consequence, the semi-norm $[z;\cdots]$ is bounded by the full Strichartz norm 
\EQ{ \label{[z] bd Stz}
 \|u\|_{[z;T_0,T_1;T_2]} \lec \|u\|_{\Stz^1(T_0,T_1)},}
since we have for any $S,T\in(T_0,T_1)$, using the Strichartz estimate for the linearized equation,
\EQ{
 \pt\|u[z,T]_>-u[z,S]_>\|_{\ST(S,T_2)}
 \pr\le \|u-u[z,S]\|_{\ST(S,T)}+\|u[z,T]-u[z,S]\|_{\ST(T,T_2)}
 \lec \|u\|_{\Stz^1(T_0,T_1)}.}

We also need non-admissible Strichartz estimates. 
\begin{lem} \label{lem:nonad}
Let $(p_0,q_0),(p_1,q_1)\in(1,\I)\times(2,6]$ and $\s_j:=2/p_j+3(1/q_j-1/2)$ satisfy 
\EQ{
 \s_0+\s_1=0>\s_j-1/p_j, \pq |\s_j|\le 2/3.}
Then there exists $C>0$ such that under the assumptions of the above lemma, 
\EQ{
 \|\D P_c f[z,s_0]\|_{L^{p_0}_tL^{q_0}_x(I)} \le C\|f\|_{L^{p_1'}_tL^{q_1'}_x(I)}.} 
If $(p_0,q_0)=(4,6)$, then for $T_0<T\le T_1$, 
\EQ{
 \|\D P_c f[z,s_0]\|_{[z;T_0,T;T_1]}\le C\|f\|_{L^{p_1'}_tL^{q_1'}_x(T_0,T)}.}
\end{lem}
\begin{proof}
This set of estimates for the free Schr\"odinger equation was proved by Kato \cite{K} for $q_j<2^*$, and by Foschi \cite{F} and Vilela \cite{V} for $q_j=2^*$. It is transferred to the time independent equation $e^{itH}P_c$ by Yajima's argument of bounded wave operators \cite{Y}. 
More precisely, the condition $|\s_j|\le 2/3$ is not needed in Kato, but only in the double endpoint case $q_0=q_1=6$ by Foschi and Vilela, in the form $p_0=p_1\in[6/5,6]$. The above lemma is infected by this condition including non-endpoint cases, because we use the double endpoint estimate to treat the time-dependent potential part as a small perturbation. 
Let $u:=\D P_c f[z,s_0]$. Then by the above estimates on $e^{itH}P_c$, 
\EQ{
 \|u\|_{(L^{p_0}_tL^{q_0}_x\cap L^{p_2}_tL^6_x)(I)} \lec \|B[z]u\|_{L^{p_2}_tL^{6/5}_x(I)}+\|f\|_{L^{p_1'}_tL^{q_1'}_x(I)},}
where $p_2:=\frac{2}{\s_0+1}\in[6/5,6]$. The potential term is bounded by 
\EQ{
 \|B[z]u\|_{L^{p_2}_tL^{6/5}_x(I)} \lec \|z\|_{L^\I_t(I)}^2\|u\|_{L^{p_2}_tL^6_x(I)} \ll \|u\|_{L^{p_2}_tL^6_x(I)},}
and thus we obtain the desired result as in the previous lemma. 
\end{proof}

\section{Linearized profile decomposition}
Now we develop a profile decomposition for the linearized equation of the radiation part in \eqref{eq zxi}. For that purpose, we need a similar notation to the above for sequences. 
For any sequences $a,b,c\etc$ the sequence in the form 
\EQ{
 \N\ni n\mapsto \sX(a_n,b_n,c_n\etc)}
for {\it any expression} $\sX$ (as long as it is well defined),  is denoted by 
\EQ{ \label{conve seq}
 \sX(a,b,c\etc):=\{\sX(a_n,b_n,c_n\etc)\}_n.}
The same convention applies when the sequence is defined only for large $n\in\N$. 
When $\sX=\{\sX_n\}_n$ is a sequence of sets, then it is regarded as the product set:
\EQ{ \label{conve prod}
 x \in \sX \iff \forall n\in\N,\ x_n\in \sX_n.}
The same convention is applied to $\lim$, $\sup$, etc., for any sequence $X=\{X_n\}_n$:
\EQ{ \label{conve lim}
 \pt X\to \lim X:=\lim_{n\to\I}X_n,\pq \Sup X:=\sup_n X_n, 
 \pr \limsup X:=\limsup_{n\to\I}X_n,\pq \liminf X:=\liminf_{n\to\I}X_n,}
unless the limit is explicitly associated with another parameter. ``$\Sup$" is capitalized to avoid possible confusion. 

The set of open intervals is denoted by 
\EQ{ \label{def sI}
 \pt \sI:=\{(a,b)\subset\R \mid a<b\}.}
For any $I\in\sI^\N$, the set $\SBC(I)$ of sequences of uniformly small, bounded and continuous functions is defined by the following. For any $z\in C(I;\C)$ 
\EQ{ \label{def SBC}
 z\in \SBC(I)} 
iff $\sup_{n\in\N}\sup_{t\in I_n}|z_n(t)|\ll 1$ and 
\EQ{ \label{SBC}
 \forall\e>0,\ \exists\de>0,\ \forall s,\forall t\in I,\pq \Sup|s-t|<\de \implies \Sup|z(s)-z(t)|<\e.}
The smallness requirement can be determined by $V$. Since it can be fixed throughout the paper but does not play any role, we leave it unspecified. 

For any $I\in\sI^\N$, any $z\in C(I;\C)$ and any $z_\I\in C(\R;\C)$, we say that $z\to z_\I$ \U{locally uniformly on $I$} iff for all $0<T<\I$ 
\EQ{ \label{def LUC}
 \lim_{n\to\I} \sup_{t\in[-T,T]\cap I_n}|z_n(t)-z_\I(t)|=0.}

Let $\sI^\N\ni I\ni s$, $C(I;\C)\ni z$ and $H^1\{s\}\ni u$. Note that they are all sequences by the above convention. Suppose that $\lim P_cu(s)$ is weakly convergent in $H^1$. Then the sequence $v \in C(I;H^1)$ solving
\EQ{
  \forall n\in\N,\pq i\dot v_n + Hv_n = B[z_n]v_n \text{ (on $I_n$)}, \pq v_n(s_n)=\lim_{k\to\I}P_cu_k(s_k) }
is denoted by 
\EQ{ \label{def limlin}
  u[z,s]\II:=\{v_n\}_n=\BR{\lim_{k\to\I}P_cu_k(s_k)}[z,s].}

Let $s'\in I$ be another sequence of times, then 
\EQ{
 u[z,s]\II[z,s']=u[z,s]\II.}

In the autonomous case $z\equiv 0$, the above object can be defined by translation: 
\EQ{
 \forall t\in\R,\pq u[0,s]\II(t)=e^{i(t-s)H}\lim_{k\to\I}P_cu_k(s_k),}
which trivializes the limiting behavior as $n\to\I$. 
The presence of $z$ disables such a precise description. However, we do not need so much to prove the scattering result, but uniform integrability in the Strichartz norms will suffice, which is given by the following lemma.

\begin{lem} \label{lem:Stzlim}
Let $s\in I\in \sI^\N$, $z\in\SBC(I)$ and $u\in P_cH^1\{s\}$ be sequences such that $\Sup\|u(s)\|_{H^1}<\I$. Then after extracting a subsequence, there exist $\fy\in P_cH^1$ and $z_\I\in C(\R;\C)$ such that $u(s)\weakto\fy$ weakly in $H^1$, $z(s+t)\to z_\I$ locally uniformly on $I-s$, and 
\EQ{ \label{aproxprof}
 \|u[z,s]\II-\fy[z_\I,0](t-s)\|_{\Stz^1(I)}\to 0.}
Moreover, for any $0<T<\I$, 
\EQ{ \label{profaprox}
 \|u[z,s]-u[z,s]\II\|_{L^\I_t(|t-s|<T;L^4_x)}\to 0.}
If the convergence $u(s)\to\fy$ is strong in $H^1$, then 
\EQ{ \label{profaproxstrong}
 \|u[z,s]-u[z,s]\II\|_{\Stz^1(I)}\to 0.}
\end{lem}
\begin{proof}
First, the uniform continuity of $\SBC$ allows us to extend $z_n$ to $\R$ so that we may assume $I=\R^\N$ without losing generality. 
The uniform boundedness in $H^1$ allows us to pass to a subsequence such that $u(s)\weakto\exists\fy$. Let 
\EQ{
 \pt t\mapsto \z(t):=z(t+s)\in \SBC(\R^\N),
 \pr t\mapsto v(t):=\fy[\z,0](t)=u[z,s]\II(t+s) \in C(\R;P_cH^1)^\N.} 
Since $\z\in\SBC$, Ascoli-Arzela implies that, after extracting a subsequence, $\z\to\exists\z_\I$ in $C(\R;\C)$ with $\|\z_\I\|_{L^\I(\R)}\ll 1$. Let $v_\I=\fy[z_\I,0]\in C(\R;P_cH^1)$. Then we have 
\EQ{
 i\dot v+Hv=B[\z]v, \pq i\dot v_\I+Hv_\I=B[\z_\I]v_\I,
 \pq \forall n,\ v_n(0)=\fy=v_\I(0).}
We have the full Strichartz estimates on $v_\I$ by Lemma \ref{lem:Stz}, and 
\EQ{
 \lim_{T\to\I}\|v_\I\|_{L^2(|t|>T;B^1_{6,2})}=0.}
Since $\z\to\z_\I$ in $L^\I(|t|\le T)$ for any $T<\I$, $B[\z]\to B[\z_\I]$ in the operator norm of $B^1_{6,2}\to B^1_{6/5,2}$, uniformly on $|t|\le T$. Hence by the Strichartz estimate on $e^{itH}P_c$, and using $H^1_{6/5}\subset B^1_{6/5,2}$, 
\EQ{
 \|v-v_\I\|_{\Stz^1(|t|<T)} \pt\lec \|B[\z]v-B[\z_\I]v_\I\|_{L^2_t(|t|<T;H^1_{6/5})}
 \pr\lec \|\z-\z_\I\|_{L^\I(|t|<T)}\|v_\I\|_{L^2_t(|t|<T;H^1_6)}
 \prQ + \|(\z,\z_\I)\|_{L^\I(|t|<T)}\|v-v_\I\|_{L^2_t(|t|<T;H^1_6)}.}
Thus, the last term being absorbed by the left, we obtain
\EQ{
 \|v-v_\I\|_{\Stz^1(|t|<T)} \to 0.} 
Applying the same estimate to the Duhamel with $e^{itH}P_c$ from $t=\pm T$, we obtain 
\EQ{
 \pt\|v-v_\I\|_{\Stz^1(|t|>T)} 
 \pr\lec \|v(\pm T)-v_\I(\pm T)\|_{H^1_x} + \|B[\z]v-B[\z_\I]v_\I\|_{L^2_t(|t|>T;H^1_{6/5})}
 \pr\lec o(1) + \|(\z,\z_\I)\|_{L^\I(|t|>T)}(\|v_\I\|_{L^2_t(|t|>T;H^1_6)}+\|v-v_\I\|_{L^2_t(|t|>T;H^1_6)}),}
where the last term is absorbed by the left. Thus we obtain
\EQ{
 \limsup \|v-v_\I\|_{\Stz^1(\R)}
 \ll \|v_\I\|_{L^2_t(|t|>T;H^1_6)}.}
Sending $T\to\I$, we see that the right side is zero, namely \eqref{aproxprof}. 

To prove \eqref{profaprox}, let $w(t):=u[z,s](t+s)$ and $\ga:=(v-w)[0,0]$. Then by the Strichartz estimate on $e^{itH}P_c$, 
\EQ{
 \pt\|(v-w)-\ga\|_{(L^\I_t L^2_x\cap L^2_tL^6_x)(|t|<T)}
 \pr\lec \|B[\z](v-w)\|_{L^2_tL^{6/5}_x(|t|<T)}
 \ll \|v-w\|_{L^2_tL^6_x(|t|<T)},}
so it is bounded by $\|\ga\|_{L^2_tL^6_x(|t|<T)}$, which tends to $0$, because $\ga\to 0$ in $L^\I_tL^4_x(|t|<T)$ and bounded in $L^2_tH^1_6$. Interpolating with the uniform bounds on $v$ and $w$, we get $v-w-\ga\to 0$ also in $L^\I_tL^4_x(|t|<T)$, hence for $v-w$ as well. 

The proof of \eqref{profaproxstrong} is similar but easier. We add one derivative to the Strichartz norms and extend to the real line, such as $L^2_tH^1_6(\R)$. Then $\ga=e^{itH}\ga(0)\to 0$ in this norm, since $\ga(0)\to 0$ strongly in $H^1$. So $v-w$ is also vanishing. 
\end{proof}

The linearized equation does not preserve either the mass or the energy, because $B[z]$ is not even $\C$-linear, but the next lemma suffices for the profile decomposition. Since $H>0$ on $P_cH^1$, its fractional power is defined. For any $\te\in[0,1]$, the inner product is defined by 
\EQ{ \label{def Hte}
 \bH_\te(\fy,\psi):=\frac12\LR{H^\te P_c\fy|\psi}, \pq \bH_\te(\fy):=\bH_\te(\fy,\fy),}
such that $\bM(\fy)=\bH_0(\fy)$ and $\bH(\fy)=\bH_1(\fy)$ for all $\fy \in P_cH^1$. 
\begin{lem} \label{lem:uniforth}
Let $s\in I\in\sI^\N$ and $z\in\SBC(I)$. Let $v^0,v^1\in C(I;P_cH^1)$ be two sequences of linearized solutions, i.e.~$v^j=v^j[z,s]$ for $j=0,1$. Suppose that $v^0(s)$ strongly converges in $H^1$ and that $v^1(s)\weakto 0$ weakly in $H^1$. Then 
\EQ{
 \forall\te\in[0,1],\pq \|\bH_\te(v^0,v^1)\|_{L^\I_t(I)} \to 0.}
\end{lem}
\begin{proof}
It suffices to show $\bH_\te(v^0,v^1)(s')\to 0$ for any $s'\in I$, along a subsequence. We use the unitarity 
\EQ{
 \ti v^j:=e^{i(s-t)H} v^j \implies \bH_\te( v^0, v^1)=\bH_\te(\ti v^0,\ti v^1)}
and the Duhamel formula 
\EQ{
 \pt \ti v^j(t) =  v^j(s)+\D^j(t), 
 \pr \D^j(t):= \int_{0}^{t-s}e^{-it'H}B[z(s+t')] v^j(s+t')dt'.}
$ v^0(s)\to\exists\fy$ strongly in $H^1$ by the assumption. Extracting a subsequence, we may assume $s'-s\to\exists s'_\I\in[-\I,\I]$ and $z(t+s)\to\exists z_\I(t)$ locally uniformly on $I-s$. Then 
\EQ{
 \D^0(s')\to \int_0^{s'_\I}e^{-itH}B[z_\I]\fy[z_\I,0]dt,}
strongly in $H^1$, by Lemma \ref{lem:Stzlim}, \eqref{aproxprof}, after passing to a subsequence. For $\D^1$, we use the $L^p$ decay estimate on $e^{itH}P_c$. Fix $0<\de\ll 1$ such that $1/q_\pm := 1/6 \pm \de \in [1/\pp,1/2)$. Then for any $0<T<\I$, 
\EQ{
 \|\D^1\|_{L^\I_t(I;L^{q_+}+L^{q_-})_x}
 \pt\lec \int_{|t-s|>T}|t-s|^{-1-3\de}\|B[z] v^1(t)\|_{L^{q_-'}_x}dt
 \prQ+ \int_{|t-s|<T}|t-s|^{-1+3\de}\|B[z] v^1(t)\|_{L^{q_+'}_x}dt
 \pr\lec T^{-3\de}\| v^1\|_{L^\I_tL^2_x}+\| v^1\|_{L^\I_t(|t-s|<T;L^4_x)},}
and the last term is vanishing by Lemma \ref{lem:Stzlim}, \eqref{profaprox}. Since $T>0$ is arbitrary, we deduce that $\D^1(s')\weakto 0$. Hence $\ti v^1(s')\weakto 0$, while $\ti v^0(s')$ is strongly convergent. Thus we obtain $\bH_\te( v^0, v^1)(s')\to 0$ as desired. 
\end{proof}

Using the above lemmas, we are ready to prove the profile decomposition for the linearized equation for $\x$.
\begin{lem} \label{lem:Lprof}
Let $s\in I\in\sI^\N$ and $z\in\SBC(I)$. Let $\psi\in(P_cH^1)^\N$ be a bounded sequence. Then passing to a subsequence, there exist $J^*\in\N\cup\{\I\}$ and $s^j\in I$ for each $\N_0\ni j<J^*$ with the following properties. Let $\nu:=\psi[z,s]\in C(I;P_cH^1)$. 
\begin{enumerate}
\item $s^0=s$ and $|s^j-s^k|\to\I$ for each $j\not=k<J^*$.
\item For $j<J^*$, $\nu(s^j)\weakto\exists\fy^j_\I\in H^1$ weakly. Put $\la^j:=\nu[z,s^j]\II=\fy^j_\I[z,s^j]$. Then $\la^j(s^k)\weakto 0$ weakly in $H^1$ for $j\not=k$, and $\fy^j_\I\not=0$ for $j>0$. 
\item For each finite $J\le J^*$, put $\ga^{J}:=\nu-\sum_{0\le j<J} \la^j$. For $j< J$, $\ga^{J}(s^j)\weakto 0$ weakly in $H^1$, and for all $\te\in[0,1]$, 
\EQ{ \label{energy decop}
 \sum_{0\le j<J}\|\la^j\|_{L^\I_t(I;\dot H^\te_x)}^2 + \|\ga^{J}\|_{L^\I_t(I;\dot H^\te_x)}^2 \sim \|\psi\|_{\dot H^\te}^2+o(1).}
$\bH_\te(\la^j,\la^k)$, $\bH_\te(\la^j,\ga^J)\to 0$ for $k\not=j< J$ and $\te\in[0,1]$, uniformly on $I$. 
\item For $0\le\te<1$, 
\EQ{ \label{Stz vanish}
 \lim_{J\to J^*}\limsup\|\ga^J\|_{[L^\I_tL^4_x,\Stz^1]_\te(I)}=0.}
\end{enumerate}
\end{lem}
We call the decomposition given by the above lemma 
\EQ{ \label{lin decop}
  \psi[z,s]= \sum_{0\le j<J} \la^j + \ga^J, \pq \la^j:=\psi[z,s][z,s^j]\II,}
\U{the linearized profile decomposition}. 
\begin{proof}
The sequence $s^j\in I$ is defined inductively as follows. 
First, let $s^0:=s$. Passing to a subsequence, we may assume $\nu(s^0)\weakto\exists\fy^0_\I$. Then $\la^0:=u[z,s^0]\II$ and $\ga^1:=\nu-\la^0$ are defined with $\ga^1(s^0)\weakto 0$. The boundedness (in $H^1$) of $\psi$ implies that $\nu$ and $\la^0$ are uniformly bounded, so is $\ga^1$. 

Let $J\in\N$ and suppose that uniformly bounded $\ga^{J}\in C(I;H^1)$ and $s^j\in I$ have been defined for $j<J$ such that $\ga^{J}(s^j)\weakto 0$. Since $\|\ga^{J}\|_{L^\I_t(I;L^4_x)}$ is bounded, we can define $s^{J}\in I$ such that 
\EQ{ \label{ga L4 limit}
 \|\ga^{J}\|_{L^\I_t(I;L^4_x)} = \|\ga^{J}(s^{J})\|_{L^4_x} + o(1).}
If the left hand sequence tends to $0$, put $J^*=J$ and the definition is terminated. 
Otherwise, put $\la^{J}:=\ga^{J}[z,s^{J}]\II$ and $\ga^{J+1}:=\ga^{J}-\la^{J}$. Passing to a subsequence, we may assume $\ga^{J}(s^{J})\weakto\exists\fy^J_\I\not=0$ in $H^1$ and $z(t+s^{J})\to\exists z_\I^{J}(t)\in C(\R;\C)$ locally uniformly on $I-s^J$. Then $\ga^{J+1}(s^{J})\weakto 0$ is obvious. Since $\ga^{J}(s^j)\weakto 0$ for $j<J$,  we have $\ga^{J}\to 0$ in $L^\I_t(|t-s^j|<T;L^4_x)$ for any $T<\I$, by Lemma \ref{lem:Stzlim}, \eqref{profaprox}. Hence $|s^J-s^j|\to\I$. 
Then from Lemma \ref{lem:Stzlim}, \eqref{aproxprof}, together with the Strichartz bound on $\fy^J_\I[z_\I^J,0]$, we deduce that $\la^J(s^j)\weakto 0$, and so $\ga^{J+1}(s^j)\weakto 0$. The same argument implies that $\la^j(s^J)\weakto 0$ as well. 

Hence we can iterate the same procedure. In this way, after the diagonalization argument, we obtain the sequences $s^j$ with the properties that $|s^j-s^k|\to\I$ for $j\not=k$, $\ga^J(s^j)\weakto 0$ for $j< J$, $\la^j(s^k)\weakto 0$ for $j\not=k$, $z(t+s^j)\to z_\I^j(t)$, and the decomposition \eqref{lin decop}.

Since $\la^j(s^j)=\fy^j_\I$ while $\la^k(s^j)\weakto 0$ for $k\not=j$, Lemma \ref{lem:uniforth} implies $\bH_\te(\la^j,\la^k)\to 0$, $\bH_\te(\la^j,\ga^k)\to 0$ for $j<k$ and $\te\in[0,1]$. Hence 
\EQ{
 \pt \bH_\te(\psi)=\bH_\te(\sum_{0\le j<J}\la^j(s)+\ga^J(s))
 =\sum_{0\le j<J}\bH_\te(\la^j(s))+\bH_\te(\ga^J(s))+o(1).}
The equivalence $\bH_\te(\fy)\sim\|\fy\|_{\dot H^\te}^2$ on $P_cH^1$ and Lemma \ref{lem:Stz} imply \eqref{energy decop}. 

It remains to prove \eqref{Stz vanish}. By the definition of $s^j$, cf.~\eqref{ga L4 limit}, we have 
\EQ{
 \|\ga^{J}\|_{L^\I_tL^4_x} = \|\fy^J_\I\|_{L^4_x}+o(1) \lec \|\la^J\|_{L^\I_t H^1_x}+o(1).}
Since the right hand side is vanishing by \eqref{energy decop},
\EQ{
 \lim_{J\to J^*}\limsup\|\ga^J\|_{L^\I_tL^4_x}=0,}
and then by interpolation with the uniform Strichartz bound, we obtain \eqref{Stz vanish}. 
\end{proof}

\section{Nonlinear perturbation estimates}
In order to use the linearized profile decomposition to approximate the nonlinear solutions, we need a few perturbation lemmas for the nonlinear equation of $\x$
\EQ{ \label{eqxi}
 i\dot \x + H\x = B[z]\x + \ti N(z,\x)}
regarding $z$ as a given time-dependent function. 
Since our global knowledge on $z$ is very poor (cf.~Section \ref{ss:diff}), we should avoid perturbing $z$ for long time. 
It leads us to prepare the following two lemmas for perturbation: Lemma \ref{lem:over river} for long time intervals where $\x$ is small, and Lemma \ref{lem:over mountain} for bounded time intervals where $\x$ may be large. 

The first lemma is a perturbation of $0$, or construction of dispersed solutions.
\begin{lem} \label{lem:smallsol}
Let $-\I<T_0<T_1\le\I$, $z\in C([T_0,T_1);\C)$ and $\fy\in H^1(\R^3)$. Put
\EQ{
 \cN_\te:=\|z\|_{L^\I(T_0,T_1)}+\|\fy\|_{H^\te}}
for $\te\in[0,1]$ and assume $\cN_0\ll 1$. 

\rm{(I)} If $\|\fy[z,T_0]\|_{\ST(T_0,T_1)}\cN_{1/2}^3\ll 1$, then \eqref{eqxi} has a unique solution $\x$ satisfying 
\EQ{ \label{small sol cond}
 \x\in C([T_0,T_1);H^1),\pq \x(T_0)=\fy, \pq \|\x\|_{\ST(T_0,T_1)}\lec\|\fy[z,T_0]\|_{\ST(T_0,T_1)}.} 

\rm{(II)} Let $\x\in C([T_0,T_1);H^1)$ be a solution of \eqref{eqxi} with $\x(T_0)=\fy$ satisfying 
\EQ{
 \|\x\|_{\ST(T_0,T_1)}\cN_{1/2}^3\ll 1.} 
Then for any $T\in(T_0,T_1)$ and all $\te\in[0,1]$,
\EQ{
  \|\x[z,T]_>-\fy[z,T_0]\|_{\Stz^\te(T_0,T_1)} \lec \cN_\te\|\fy[z,T_0]\|_{\ST(T_0,T)}(\cN_0+\|\fy[z,T_0]\|_{\ST(T_0,T)}),} 
and $\|\x\|_{[z;T_0,T;T_1]} \ll \|\x[z,T_0]\|_{\ST(T_0,T)} \sim \|\x\|_{\ST(T_0,T)}$. 
\end{lem}
\begin{proof}
Let $\x_0:=\fy[z,T_0]$ and $\|\x_0\|_{\ST(T_0,T_1)}=:\de_0$. The solution $\x$ is obtained by the iteration argument. If $\x$ is a solution, then by the Strichartz estimate: Lemma \ref{lem:Stz},
\EQ{  \label{est by Stz}
 \|\x-\x_0\|_{\Stz^\te}
 \lec \|\x\|_{\Stz^\te}\|\x\|_{\ST}(\|\x\|_{\ST}+\|z\|_{L^\I})}
for $\te\in[0,1]$. Using the non-admissible Strichartz: Lemma \ref{lem:nonad}, 
\EQ{
 \|\x\|_{\ST} \pt\le \|\x_0\|_{\ST} + C\|\ti N(z,\x)\|_{L^{8/3}_tL^{4/3}_x}
 \pr \lec \de_0 + \|\x\|_{\ST}\|\x\|_{L^8_tL^4_x}(\|\x\|_{L^\I_tL^3_x}+\|\Phi[z]\|_{L^\I_tL^3_x})
 \pr \lec \de_0 + \|\x\|_{\ST}^{3/2}\|\x\|_{\Stz^{1/2}}^{1/2}(\|\x\|_{\Stz^{1/2}}+\cN_0).}
Suppose that $\|\x\|_{\Stz^{1/2}}\le C\cN_{1/2}$ and $\|\x\|_{\ST}\le C\de_0$ for some constant $C\gg 1$ on some shorter interval. Then 
\EQ{
 \pt\|\x\|_{\Stz^{1/2}} \lec \cN_{1/2} + C^2\de_0(\de_0+\cN_0)\|\x\|_{\Stz^{1/2}},
 \pr\|\x\|_{\ST} \lec \de_0 + C^2(\de_0\cN_{1/2}^3)^{1/2} \|\x\|_{\ST}.}
Since $\de_0+\cN_0\lec\cN_{1/2}$, we have $\de_0(\de_0+\cN_0)\lec(\de_0\cN_{1/2}^3)^{1/2}$. 
Hence, if $\de_0\cN_{1/2}^3\ll 1$ then $\|\x\|_{\ST} \lec \de_0$ and $\|\x\|_{\Stz^{\te}} \lec \cN_{\te}$ for all $\te\in[0,1]$. Then by the continuity for extending the interval, these bounds holds on the whole $(0,T)$. 

If we assume $\|\x\|_{\ST(T_0,T_1)}\cN_{1/2}^3\ll 1$ instead of $\x_0$, then we obtain $\|\x_0\|_{\ST(T_0,T)}\lec\|\x\|_{\ST(T_0,T)}$ in the same way, starting from $T=T_0+0$. 
In both cases, repeating the Strichartz estimate on $H^\te$ as above, we obtain (II).
\end{proof}
\begin{rem}
Since $\dot H^{1/2}(\R^3)$ is the scaling invariant norm for the NLS without the potential, $\cN_{1/2}$ is in general large, when we use the above lemma. In the focusing case, $\cN_{1/2}\sim 1$ on $\Soli_1$, while there is no upper bound on $\cN_{1/2}$ in the defocusing case. However, we can expect that $\ST$ is small for dispersive solutions, by which the assumptions in the above lemma can be satisfied. Note also that the estimate cannot be closed if we use only the admissible Strichartz $\Stz^{1/2}$ as in \eqref{est by Stz}, because of the quadratic terms. 
\end{rem}

If the above solution is obtained for $t\to\I$, then it scatters. 
\begin{lem} \label{lem:scatxi}
Let $T_0\in\R$, $(z,\x)\in C([T_0,\I);\C\times P_cH^1)$ solve \eqref{eqxi} on $(T_0,\I)$ satisfying 
\EQ{
 \|z\|_{L^\I(T_0,\I)}+\|\x\|_{L^\I_t(T_0,\I;L^2_x)} \ll 1, \pq \|\x\|_{L^\I_t(T_0,\I;H^1_x)}<\I.}
Then the following {\rm(i)} and {\rm(ii)} are equivalent. 
\begin{enumerate}
\item $\exists\fy\in H^1(\R^3)$ such that $\|\fy[z,T_0]-\x\|_{H^1_x}\to 0$ as $t\to\I$. 
\item $\|\x\|_{\ST(T_0,\I)}<\I$.
\end{enumerate}
In this case, we say that \U{$\x$ scatters with $z$ as $t\to\I$}. Moreover, as $T\to\I$,
\EQ{
 \|\x[z,T]-\fy[z,T_0]\|_{\Stz^1(T_0,\I)}
 +\|\x-\x[z,T]\|_{\Stz^1(T,\I)} \to 0, }
and for any $\ti z\in C([T_0,\I);\C)$ satisfying $\|\ti z\|_{L^\I_t}\ll 1$, 
\EQ{ \label{scatt decay}
 \pt \|\x\|_{[\ti z;T,\I;\I]}+\|\x[\ti z,T]\|_{L^2B^1_{6,2}(T,\I)} +\|\x\|_{L^2B^1_{6,2}(T,\I)} \to 0,}
uniformly with respect to $\ti z$. A sufficient condition of scattering is 
\EQ{ \label{suf cond scat}
 \|\x[z,T_0]\|_{\ST(T_0,\I)}\BR{\|z\|_{L^\I(T_0,\I)}+\|\x(T_0)\|_{H^{1/2}_x}}^3 \ll 1.} 
\end{lem}
The scattering with $z$ as $t\to-\I$ is defined in the same way, which has the same property as above.  
\begin{proof}
Assume (i) and let $\x_+:=\fy[z,T_0]$. Then by Lemma \ref{lem:Stz}, $\|\x_+\|_{\Stz^1(T_0,\I)}<\I$ and so in particular $\|\x_+\|_{\ST(T,\I)}\to 0$ as $T\to\I$. By (i) and the Strichartz estimate, 
\EQ{
 \|\x_+-\x[z,T]\|_{\ST(T,\I)} \lec \|\x_+(T)-\x(T)\|_{H^{1/2}} \to 0.}
Hence for sufficiently large $T$, the previous lemma implies $\|\x\|_{\ST(T,\I)}\lec\|\x[z,T]\|_{\ST(T,\I)}$. 

Assume (ii) and let $T>T_0$ so large that we can apply the previous lemma on $(T,\I)$ to get $\|\x\|_{\Stz^1(T,\I)}\lec\|\x[z,T]\|_{\Stz^1(T,\I)}$. Then for any $T_2>T_1>T$, by the Strichartz estimate: Lemma \ref{lem:Stz}, 
\EQ{
 \|\x[z,T_2]-\x[z,T_1]\|_{\Stz^1(T_0,\I)} \pt\lec \|\ti N(z,\x)\|_{\Stz^{1*}(T_1,T_2)} 
 \pr\lec \|\x\|_{\Stz^1(T_1,\I)}[\|\x\|_{\ST(T_1,\I)}+\|\x\|_{\ST(T_1,\I)}^2]\to 0}
as $T_1\to\I$. In particular $\x[z,T](T_0)$ is Cauchy in $H^1_x$ as $T\to\I$, hence convergent to some $\x_+\in H^1_x$. Then 
\EQ{
 \|\x_+[z,T_0](T)-\x(T)\|_{H^1_x} \pt\le \|\x_+[z,T_0]-\x[z,T]\|_{\Stz^1(T_0,\I)}
 \pr\lec \|\x_+-\x[z,T](T_0)\|_{H^1_x} \to 0 \pq (T\to\I)}
and so (i). Thus in either case, we have $\|\x[z,T]\|_{\ST(T,\I)}\to 0$ as $T\to\I$, hence the previous lemma implies 
\EQ{
 \|\x-\x[z,T]\|_{\Stz^1(T,\I)} \lec \BR{\|z\|_{L^\I_t(T_0,\I)}+\|\x\|_{L^\I_t(T_0,\I;H^1_x)}}\|\x[z,T]\|_{\ST(T,\I)}\to 0}
as $T\to\I$. Hence $\|\x\|_{L^2_tB^1_{6,2}(T,\I)}\to 0$. 
Let $\x_0:=\x[z,T]$ and $\x_1:=\x[\ti z,T]$. Then 
\EQ{
 \x_1 = \x_0 + \D(B[\ti z]-B[z])\x_1[z,T]}
and by the Strichartz estimate: Lemma \ref{lem:Stz}, 
\EQ{
 \|\x_1\|_{\Stz^1(T,\I)} \le \|\x_0\|_{\Stz^1(T,\I)}+C\|z\|_{L^\I_t(T_0,\I)}^2\|\x_1\|_{L^2_tB^1_{6,2}(T,\I)}.} 
After the last term is absorbed by the left, we obtain, as $T\to\I$, 
\EQ{
 \|\x[\ti z,T]\|_{L^2_tB^1_{6,2}(T,\I)} \le 2\|\x[z,T]\|_{L^2_tB^1_{6,2}(T,\I)}\to 0,}
which is uniform with respect to $\ti z$. By the previous lemma, we also have 
\EQ{
 \|\x\|_{[\ti z;T,\I;\I]} \ll \|\x[\ti z,T]\|_{\ST(T,\I)} \sim \|\x\|_{\ST(T,\I)}\to 0.}
The sufficiency of \eqref{suf cond scat} for scattering is now obvious by Lemma \ref{lem:smallsol}. 
\end{proof}

The next two lemmas are concerned with difference of two solutions. 
For the sake of brevity, the following notation is introduced for differences. 
For any expressions $\sX$ and $a,b,c\etc$  
\EQ{ \label{conve diff}
 \diff{\sX(a_\pa,b_\pa,c_\pa\etc)}:=\sX(a_0,b_0,c_0\etc)-\sX(a_1,b_1,c_1\etc).}

The first lemma of difference estimates treats perturbation of dispersed solutions. It will be used either with $z_0=z_1$ or on a short interval. We need the non-admissible Strichartz for difference of quadratic terms. 
\begin{lem} \label{lem:over river}
Let $-\I<T_0<T_1<T_2\le\I$, $z_0\in C([T_0,T_2);\C)$, $z_1\in C([T_0,T_1];\C)$, and $\x_0,\x_1\in C([T_0,T_1];H^1)$ solve 
\EQ{
 \forall t\in(T_0,T_1),\pq i\dot \x_j + H \x_j = B[z_j]\x_j + \ti N(z_j,\x_j).} 
Put $\cN_\te:=\|z_j\|_{L^\I_t(T_0,T_1)}+\|\x_j\|_{L^\I_t(T_0,T_1;H^\te_x)}$ for $\te\in[0,1]$, and suppose that for some $0<\de\lec\ti\de$ satisfying $\cN_0+\ti\de\cN_{1/2}^3 \ll 1$, 
\EQ{
 \pt \|\x_0[z_0,T_0]\|_{\ST(T_0,T_1)} \le \ti\de, \pq \|\diff\x_\pa[z_0,T_0]\|_{\ST(T_0,T_1)}+ \|\diff z_\pa\|_{L^4(T_0,T_1)} \le \de.}
Then we have
\EQ{
 \|\diff \x_\pa\|_{[z;T_0,T_1;T_2]} \lec (\cN_0\cN_1)^{3/7}\ti\de^{4/7}\de^{1/7}.}
\end{lem}
\begin{proof}
The previous lemma \ref{lem:smallsol} applies to both $(z_j,\x_j)$, which implies 
\EQ{
 \|\x_j\|_{\Stz^\te(T_0,T_1)}\lec \cN_\te, \pq \|\x_j\|_{\ST(T_0,T_1)}\lec \ti\de.}
Now apply the non-admissible Strichartz estimate, Lemma \ref{lem:nonad} to Duhamel
\EQ{
 \diff\x_\pa = \diff\x_\pa[z_0,T_0] + \D\{\diff{\ti N}(z_\pa,\x_\pa)+\diff B[z_\pa]\x_1\}[z_0,T_0].}
Choose 
\EQ{
 (p_0,q_0)=(4,24/7),\ (p_1,q_1)=(2,24/5),\ (p_2,q_2)=(4,24/9)}
so that $\s_0=-1/8$, $\s_1=\s_2=1/8$ and we can apply the lemma. Then 
\EQ{
 \pt\|\diff\x_\pa - \diff\x_\pa[z_0,T_0]\|_{L^{4}_tL^{24/7}_x}
 \lec \|\diff{\ti N}(z_\pa,\x_\pa)+\diff B[z_\pa]\x_1\|_{L^{2}_tL^{24/19}_x+L^{4/3}_tL^{24/15}_x}
 \pr\lec \{\|\diff\x_\pa\|_{L^{4}_tL^{24/7}_x}+\|\diff z_\pa\|_{L^4_t}\}(\|\x_0\|_{\ST}+\|\x_1\|_{\ST})\{1+\|\x_0\|_{\ST}+\|\x_1\|_{\ST}\},}
where the linear and quadratic (in $\x_0,\x_1$) terms are estimated in $L^2_tL^{24/19}_x$, while the cubic terms are in $L^{4/3}_tL^{24/15}_x$. The factor $1$ comes from the term $\diff B[z]\x_1$, and it also includes the smallness factor $\cN_0\ll 1$.  
For the cubic term with $\diff z_\pa$, we used 
\EQ{
 \|\diff R[z_\pa]\x_j\|_{H^1_x} \lec |\diff z_\pa|\|\x_j\|_{L^2_x} \le |\diff z_\pa| \cN_0.}
Thus, using the smallness of $\|\x_j\|_{\ST(T_0,T_1)}$ and $\|\diff z_\pa\|_{L^4(T_0,T_1)}$, we obtain
\EQ{
 \|\diff\x_\pa\|_{L^4_tL^{24/7}_x(T_0,T_1)} \lec \|\diff\x_\pa[z_0,T_0]\|_{L^{4}_tL^{24/7}_x(T_0,T_1)}+\ti\de\de.}
Applying the same estimate to the Duhamel formula 
\EQ{
 \diff\x_\pa[z_0,T_1]_> = \diff\x_\pa[z_0,T_0] + \D 1_{T_0<t<T_1}\{\diff{\ti N}(z_\pa,\x_\pa)+\diff B[z_\pa]\x_1\}[z_0,T_0],}
we also obtain 
\EQ{
 \|\diff\x_\pa[z_0,T_1]_> - \diff\x_\pa[z_0,T_0]\|_{L^{4}_tL^{24/7}_x(T_0,T_2)}
 \lec \ti\de\|\diff\x_\pa[z_0,T_0]\|_{L^{4}_tL^{24/7}_x(T_0,T_1)}+\ti\de\de,}
and this norm is related to $\ST=L^4_tL^6_x$ by interpolation and Sobolev as 
\EQ{
 \pt \|f\|_{\ST} \lec \|f\|_{L^{4}_tL^{24/7}_x}^{4/7}\|f\|_{L^4B^0_{\I,2}}^{3/7}
 \lec \|f\|_{L^{4}_tL^{24/7}_x}^{4/7}\|f\|_{\Stz^1}^{3/7},
 \pr \|f\|_{L^{4}_tL^{24/7}_x} \le \|f\|_{\ST}^{1/4}\|f\|_{L^{4}_tL^{3}_x}^{3/4}
 \lec \|f\|_{\ST}^{1/4}\|f\|_{\Stz^0}^{3/4}.}
Injecting these to the above and using the Strichartz bound, we obtain
\EQ{
 \|\diff\x_\pa[z_0,T_1]_> - \diff\x_\pa[z_0,T_0]\|_{\ST(T_0,T_2)}
 \pt\lec [\ti\de\de^{1/4}\cN_0^{3/4}]^{4/7}\cN_1^{3/7} 
 \pr= (\cN_0\cN_1)^{3/7}\ti\de^{4/7}\de^{1/7},}
as desired. 
\end{proof}

The second lemma of difference estimates treats perturbation of large solutions with finite Strichartz. 
It will be used only on a bounded interval of time. 
\begin{lem} \label{lem:over mountain}
Let $-\I<T_0<T_1<T_2\le\I$, $z_0\in C([T_0,T_2);\C)$, $z_1\in C([T_0,T_1];\C)$, and $\x_0,\x_1\in C([T_0,T_1];H^1)$ solve 
\EQ{
 \forall t\in(T_0,T_1),\pq i\dot \x_j + H \x_j = B[z_j]\x_j + \ti N(z_j,\x_j) .}Put for $\te\in[0,1]$ 
\EQ{
 \cN_\te:=\|z_j\|_{L^\I_t(T_0,T_1)}+\|\x_j\|_{L^\I_t(T_0,T_1;H^\te_x)},
 \pq \cN_2:=\|\x_0\|_{\ST(T_0,T_1)},}
and assume $\cN_0\ll 1$. For any $\e>0$, there is $\de_*(\cN_1,\cN_2,\e)>0$, continuous and decreasing for each $\cN_j$, such that if
\EQ{
 \|\diff\x_\pa[z_0,T_0]\|_{\ST(T_0,T_1)}+ \|\diff z_\pa\|_{L^4(T_0,T_1)} \le \de_*}
then we have
\EQ{
 \|\diff \x_\pa\|_{[z_0;T_0,T_1;T_2]} \le \e.}
\end{lem}
\begin{proof}
For any $N\in\N$, $(T_0,T_1)$ is decomposed into subintervals $I_0\etc I_N$ such that 
\EQ{
 \forall j=0\etc N,\pq \|\x_0\|_{\ST(I_j)} \le 2N^{-1/4}\cN_2=:\ti\de.}
Let $I_j=:(S_j,S_{j+1})$. If $\ti\de\cN_{1/2}^3\ll 1$, then we can apply Lemma \ref{lem:smallsol} starting from $S_j$, and we obtain 
\EQ{
 \|\x_0[z_0,S_j]\|_{\ST(I_j)} \sim \|\x_0\|_{\ST(I_j)} \le \ti\de.}
Suppose that for some $\de_0>0$, 
\EQ{
 \|\diff\x_\pa[z_0,S_0]\|_{\ST(S_0,T_1)}+ \|\diff z_\pa\|_{L^4(S_0,T_1)} \le \de_0 \le \ti\de, \pq \ti\de \cN_{1/2}^3 \ll 1.}
Then $\|\x_1[z_0,T_0]\|_{\ST(S_0,S_1)} \lec \ti\de$ 
and we can apply Lemma \ref{lem:over river} on $I_0$. Then 
\EQ{
 \|\diff\x_\pa\|_{[z_0;S_0,S_1;T_2]}
 \le C(\cN_0\cN_1)^{3/7}\ti\de^{4/7}\de_0^{1/7}.}
Let $\de_1:=\de_0+C(\cN_0\cN_1)^{3/7}\ti\de^{4/7}\de_0^{1/7}$, then 
\EQ{
 \|\diff\x_\pa[z_0,S_1]\|_{\ST(S_1,T_1)}+ \|\diff z_\pa\|_{L^4(S_1,T_1)} \le \de_1.}
Hence if $\de_1\le\ti\de$ then we can repeat the same thing on $I_1$. 
Define the sequence $\de_j$ for $j=0\etc N$ inductively from $\de_0$ by 
\EQ{
 \de_{j+1}:=\de_j + C(\cN_0\cN_1)^{3/7}\ti\de^{4/7}\de_j^{1/7}.}
Given $\cN_1$ and $\cN_2$, we can determine $\ti\de$ and $N$ such that 
\EQ{
 2N^{-1/4}\cN_2 \le \ti\de \ll \cN_1^{-3} \le \cN_{1/2}^{-3}.} 
Then there is $\de_*=\de_*(\cN_1,\cN_2,\e)>0$ such that 
\EQ{
 \de_0\le\de_* \implies \de_{N+1} < \min(\ti\de,\e).}
Then for $\de_0\le\de_*$, we can iterate the above estimate for all $j$ to get
\EQ{
 \|\diff\x_\pa\|_{[z_0;T_0,T_1;T_2]}
 \pt\le \sum_{j=0}^N \|\diff\x_\pa\|_{[z_0;S_j,S_{j+1};T_2]} 
 \pr\le \sum_{j=0}^N C(\cN_0\cN_1)^{3/7}\ti\de^{4/7}\de_j^{1/7}=\de_{N+1}-\de_0<\e,}
where we used the subadditivity for consecutive intervals by Lemma \ref{lem:subadd}. 
\end{proof}

\section{Nonlinear profile decomposition}
We are now ready to develop a profile decomposition for NLS \eqref{NLSP} in the $(z,\x)$ coordinate, i.e. the equation \eqref{eq zxi}. 
Let $I\in\sI^\N$ and $u\in C(I;H^1[\mu_p])$ be a sequence of solutions of \eqref{NLSP} in the mass region of the $(z,\x)$-coordinate. 
Put 
\EQ{
 I_n=:(\U T_n,\ba T_n)} 
for each $n\in\N$. 
We can uniquely write $u=\Phi[z]+R[z]\x$ by Lemma \ref{lem:decop2Phi}, where $(z,\x)\in C(I;\C\times H^1)$ is a sequence of solutions for \eqref{eq zxi}. Suppose 
\EQ{ \label{asm on un}
 \cN_1:=\Sup \BR{\|z\|_{L^\I_t(I)}+\|u\|_{L^\I_t(I;H^1)}}<\I,}
where the $z$ part is uniformly bounded by $\sqrt{\mu_p}$, so that we can omit it. 
Then we have $z\in\SBC(I)$, since the smallness of $|z|$ is already in Lemma \ref{lem:decop2Phi}, while \eqref{SBC} follows from a uniform bound on $|\dot z|$ (depending on $\cN_1$), easily observed in the equation \eqref{eq zxi} using $\|\y\|_{H^1}\lec\cN_1$, $H^1_x\subset L^6_x$, and the compactness of $\ba{D_p}\subset D_b$ for the $z$ dependence.
 By the $L^2$ conservation 
\EQ{
 \cN_0:=\Sup \BR{\|z\|_{L^\I_t(I)}+\|u\|_{L^\I_t(I;L^2)}} \lec \sqrt{\mu_p} \ll 1.}
Similarly for $\te\in[0,1]$, put
\EQ{
 \cN_\te:=\Sup\BR{\|z\|_{L^\I_t(I)}+\|u\|_{L^\I_t(I;H^\te)}}
 \le \cN_0^{1-\te}\cN_1^\te.}

For any $s\in I$, we can apply the linearized profile decomposition: Lemma \ref{lem:Lprof} to the sequence $\x(s)$. Passing to a subsequence, we have for each $J<J^*$, 
\EQ{ \label{def LPD}
 \x[z,s]=\la^{[0,J)}+\ga^J, \pq  \pq \la^j:=\x[z,s][z,s^j]\II,}
where the following abbreviation is used: for any interval $I$, 
\EQ{ \label{def laI}
 \la^I:=\sum_{j\in I\cap \Z}\la^j.}
Extracting a subsequence if necessary, we may assume 
\EQ{ \label{sj-s}
 s^j-s\to\s^j\I,\pq \exists\s^j\in\{+,-\}} 
for each $0<j<J^*$, and, since \eqref{eq zxi} implies that $(z,\x)$ is weakly equicontinuous, 
\EQ{
 (z,\x)(s^j+t) \to \exists (z_\I^j,\x_\I^j)(t) \IN{\C\times\weak{H^1}},}
locally uniformly on $I-s^j$. Put
\EQ{
 I^j_\I:=\liminf(I-s^j), \pq I^j:=I^j_\I+s^j.}
\eqref{sj-s} implies $I^j_\I\supset\{-\I<\s^j t\le 0\}$ for $j>0$. For $j=0$, we have $I^0_\I\supset[0,\I)$ if $\ba T-s\to\I$ and $I^0_\I\supset(-\I,0]$ if $\U T-s\to-\I$. Note that if $|I_n|$ is bounded then the decomposition is trivial, i.e. $J^*=1$. In either case, $I^j\ni s^j,s^0=s$. By the property of $\SBC$, we can extend $z_\I^j$ to $\SBC(\R)$. 
By the subcritical nature of NLS, it is easy to see that the weak limit $(z_\I^j,\x_\I^j)$ is a solution of \eqref{eq zxi} in $C(I_\I^j;\C\times P_cH^1)$. 
In other words, 
\EQ{
 u_\I^j:=\Phi[z_\I^j]+R[z_\I^j]\x_\I^j}
is a solution of \eqref{NLSP} on $I_\I^j$. 
Using that $R[z]-1$ is compact on $H^1$, together with the weak convergence of the linearized profiles, we have 
\EQ{
 u(s)\pt=\Phi[z_\I^0(0)]+R[z_\I^0(0)]\x(s)+o(1)
 \pr=u_\I^0(0)+\la^{(0,J)}(s)+\ga^J(s)+o(1)\pq\IN{H^1},}
and, using the orthogonality as well, 
\EQ{ \label{ME decop}
 \pt \bM(u)=\bM(u_\I^0)+\sum_{0<j<J}\bM(\la^j(s))+\bM(\ga^J(s))+o(1),
 \pr \bE(u)=\bE(u_\I^0)+\sum_{0<j<J}\bH^0(\la^j(s))+\bH^0(\ga^J(s))+o(1).}

\U{The nonlinear profile} $\La^j\in C(I^j;P_cH^1)$ is defined by
\EQ{ \label{def La}
 \La^j(t):=\x_\I^j(t-s^j).}
Also put
\EQ{
 \pt \fy^j_\I:=\lim \nu(s^j)=\la^j(s^j)\in P_cH^1,\pq \la_\I^j:=\fy^j_\I[z_\I^j,0]\in C(\R;P_cH^1),
 \pr z_j(t):=z(s^j+t), \pq z\II^j(t):=z_\I^j(t-s^j), \pq \la\II^j(t):=\la_\I^j(t-s^j),}
such that $z_j=z_\I^j+o(1)$ in $C(\R)$, and using Lemma \ref{lem:Stzlim}, \eqref{aproxprof},  
\EQ{ \label{zej stz}
 \la^j=\fy^j_\I[z,s^j]=\la\II^j+o(1) \IN{\Stz^1(I)},}
while $(z\II^j,\La^j)$ is a sequence of solutions of \eqref{eq zxi} on $I^j$. 

\U{The nonlinear remainder} $\Ga^J\in C(I;H^1)$ is defined for each $J$ by the same sequence of equations as $\x$, with the same initial data as $\ga^J$:
\EQ{ \label{eq Ga}
 i\dot \Ga^J + H \Ga^J = B[z]\Ga^J + \ti N(z,\Ga^J), \pq \Ga^J(s)=\ga^J(s).}
Since $\Lim_{J\to J^*}\limsup\|\ga^J\|_{\ST(I)}=0$, Lemma \ref{lem:smallsol} ensures the unique existence of $\Ga_n^J$ for large $J$ and large $n$, satisfying 
\EQ{ \label{Ga bd}
 \|\Ga_n^J\|_{\Stz^\te(I_n)}\lec \cN_\te, \pq \|\Ga_n^J-\ga_n^J\|_{\ST(I_n)} + \|\Ga_n^J\|_{[z_n;s_n,\ba T_n;\ba T_n]} \ll\|\ga_n^J\|_{\ST(I_n)}.}
We have $\ga^J(s^j)\weakto 0$ for each $j< J$. 
Moreover, for large $J$ and for $0\le j<J$, 
\EQ{ \label{Ga vanish st}
 \Ga^J(s^j)\weakto 0,\pq 0<\forall \ta<\I, \pq \|\Ga^J\|_{\ST(|t-s^j|<\ta)}+\|\ga^J\|_{\ST(|t-s^j|<\ta)}\to 0.}
\begin{proof}[Proof of \eqref{Ga vanish st}]
Let $X^j$ be a sequence of weighted norms on $I$ defined by
\EQ{
 \|f\|_{X^j} := \sup_{t\in I} \LR{t-s^j}^{-\de}\|f(t)\|_{(L^4+L^\pp)_x(\R^3)},}
where $\de>0$ is fixed such that $0<\de<1/2-3/\pp$. 
By Lemma \ref{lem:Stzlim}, \eqref{profaprox}, we have $\|\ga^J(t)\|_{L^4_x}\to 0$ locally uniformly around $t=s^j$, which implies $\|\ga^J\|_{X^j}\to 0$, thanks to the decaying weight. 
Suppose that $\s^j=+$, namely $s^j-s^0\to\I$. Put $F^J:=B[z]\Ga^J+\ti N(z,\Ga^J)$. Then by the $L^p$ decay estimate on $e^{itH}P_c$, we have at any $t\in I_n$ satisfying $t>s_n$, 
\EQ{ \label{weighted Ga est}
  \|\Ga_n^J(t)-\ga_n^J(t)\|_{L^4_x+L^\pp_x}
 \pt\lec \int_{s_n^0}^{t}f(t-t')\|F_n^J(t')\|_{L^{4/3}_x\cap L^{\pp'}_x}dt',
 \prq f(t):=\min(|t|^{-3(1/2-1/\pp)},|t|^{-3/4}), }
and, by H\"older and Sobolev, 
\EQ{
 \|F_n^J\|_{L^{4/3}_x\cap L^{\pp'}_x} \pt\lec \|\Ga_n^J\|_{L^4_x+L^\pp_x}(|z_n|^2+\|\Ga_n^J\|_{L^4_x\cap L^2_x}^2)
 \pr\lec \{\cN_0^2+\|\Ga_n^J\|_{L^\I_tL^4_x(I_n)}^2\}\|\Ga_n^J\|_{L^4_x+L^\pp_x}.}
By Lemma \ref{lem:smallsol}, we have 
\EQ{
 \|\Ga_n^J\|_{L^\I_tL^4_x(I_n)} \le \|\ga^J_n\|_{L^\I_tL^4_x(I_n)}+C\cN_{3/4}\|\ga^J_n\|_{\ST(I_n)}(\cN_0+\|\ga_n^J\|_{\ST(I_n)}) }
Taking $J$ and $n$ larger if necessary, we may assume that the right hand side is bounded by $\cN_0 \ll 1$. Inserting this to the above estimate and then to \eqref{weighted Ga est} yields
\EQ{
 \pn\|\Ga_n^J(t)-\ga_n^J(t)\|_{L^4_x+L^\pp_x} 
 \pt\lec \cN_0^2 \|\Ga_n^J\|_{X_n^j}\int_{\R}f(t-t')\LR{t'-s_n^j}^{\de}dt'
 \pr\lec \cN_0^2\|\Ga_n^J\|_{X_n^j}\LR{t-s_n^j}^{\de},} 
where in estimating the integral in $t'$, our choice of $\de$ implies $-3(1/2-1/\pp)+\de<-1$, which ensures the integrability for $t'\to\pm\I$. 
The estimate in the cases $t<s_n$ and $\s^j=-$ is the same, as well as for $j=0$. Thus we obtain
\EQ{
 \lim\|\Ga^J\|_{X^j} \lec \lim\|\ga^J\|_{X^j}= 0,}
for large $J$. 
By the uniform $H^1$ bound, this is equivalent to $\Ga^J(t)\weakto 0$ weakly in $H^1$, locally uniformly around $s^j$. Interpolation with \eqref{Ga bd} yields the other part. 
\end{proof}

Let us now concentrate on the estimate on the time interval $t>s_n$, assuming 
\EQ{
  \ba T_n-s_n \to\I,\pq(n\to\I)}
since otherwise uniform Strichartz bound for $\x_n$ on $t>s_n$ is trivial. The restriction to $t>s_n$ allows us to ignore the profiles with $s^j-s^0\to-\I$. 

Fix a finite $J\le J^*$, so large that \eqref{Ga vanish st} holds. 
After neglecting those profiles with $s^j-s^0\to-\I$, and reordering the profiles\footnote{This reordering can not be performed before fixing $J<\I$, since more and more linear profiles may well appear between the previous profiles as $J\to\I$, which is a typical dispersive behavior of the remainder $\ga^J$.},  we may assume for $0<j<J$
\EQ{
 s^j-s^{j-1}\to\I.}
Since $J$ is now fixed, we can no longer gain a small factor by sending $J\to\I$. Instead another parameter $0<\ta\to\I$ is introduced, decomposing the time intervals
\EQ{
 (s,\ba T)=\Cu_{0\le j<J}(s^j_-,s^j_+)\cup(s^j_+,s^{j+1}_-),} 
where $s^j_\pm\in\R^\N$ are defined for each $j$ by 
\EQ{ \label{def sjpm}
 \pt s^j_-:=\max(s^j-\ta,s), 
 \pq s^j_+:=\min(s^j+\ta,\ba T), \pq  s^J_\pm:=\ba T.}
Henceforth, $o_\ta$ denotes any sequence of real numbers satisfying
\EQ{ \label{def ota}
 \pt X(\ta) = o_\ta \iff \lim_{\ta\to\I}\limsup_{n\to\I} X_n(\ta)=0.}
By the uniform integrability \eqref{zej stz} of the linearized profiles, and their separation $s^{j}-s^{j-1}\to\I$, we have for $0\le j< J$ 
\EQ{ \label{zej st away}
 \pt  \sum_{j<k<J} \|\la^k\|_{\ST(s^0,s^j_+)}+\sum_{0\le k<j}\|\la^k\|_{\ST(s^j_-,\ba T)}=o(1),
 \pr \|\la^j\|_{\ST(s^0,s^j_-)}+ \|\la^j\|_{\ST(s^j_+,\ba T)}=o_\ta.}
The following is the main property of the nonlinear profile decomposition.
\begin{lem} \label{lem:Nprofile}
In the above setting, let $0\le l\le J$ and suppose that $\x_\I^j$ is scattering with $z^j_\I$ as $t\to\I$ for $0\le j<l$. Let $\ell:=\min(l,J-1)$. Then 
\begin{enumerate}
\item For $0\le j<J$, we have 
\EQ{
 \|\x-\La^j\|_{[z;s^j_-,s^j_+;\ba T]}+\|\Ga^J\|_{[z;s^j_-,s^j_+;\ba T]}=o(1).}
\item For $0\le j\le\ell$, we have  
\EQ{ \label{laga before}
 \|(\x-\Ga^J)[z,s^j_-]-\la^{[j,J)}\|_{\ST(s^j_-,\ba T)}+\|\La^j[z,s^j_-]_>-\la^j\|_{\Stz^1(I)}=o_\ta,}
\EQ{ \label{La before}
 \|\La^j\|_{[z;s^0,s^j_-;\ba T]}=o_\ta.}
\item For $0\le j<l$, we have 
\EQ{
 \|\x-\Ga^J\|_{[z;s^j_+,s^{j+1}_-;\ba T]}+\|\La^j\|_{[z;s^j_+,\ba T;\ba T]}=o_\ta.}
\item For $0<j\le\ell$, we have $\|\x^j_\I-\la^j_\I\|_{\Stz^1(-\I,-\ta)}\to 0$ as $\ta\to\I$. In other words, $\x^j_\I$ scatters as $t\to-\I$ and the scattering profile is $\la^j_\I$. 
\end{enumerate}
Moreover, $\x$ is bounded in $\Stz^1(s,s^l_+)$. 
\end{lem}
\begin{proof}
For the first term of (i), the locally uniform convergence of $(z,\x)(t+s^j)\to(z^j_\I,\x^j_\I)$ implies, using Lemma \ref{lem:Stzlim}, \eqref{profaprox}, 
\EQ{
 \|z-z^j_{(\I)}\|_{L^4_t(s^j_-,s^j_+)}+\|\x[z,s^j_-]-\La^j[z,s^j_-]\|_{\ST(s^j_-,s^j_+)}=o(1).}
Then by Lemma \ref{lem:over mountain}, we obtain 
\EQ{
 \|\x-\La^j\|_{[z;s^j_-,s^j_+;\ba T]}=o(1).}
For the second term of (i), using \eqref{use wave norm}, Lemma \ref{lem:smallsol} and \eqref{Ga vanish st}, we obtain   
\EQ{ \label{Ga s-2s+}
 \|\Ga^J[z,s^j_-]-\Ga^J[z,s^j_+]\|_{\ST(s^j_+,\ba T)}
 \le \|\Ga^J\|_{[z;s^j_-,s^j_+;\ba T]} \ll \|\Ga^J\|_{\ST(s^j_-,s^j_+)}=o(1).}

\eqref{La before} follows from \eqref{laga before}, since using \eqref{ker [z]} and \eqref{[z] bd Stz}, we have
\EQ{
 \|\La^j\|_{[z;s^0,s^j_-;\ba T]} \pt\le \|\La^j-\la^j\|_{[z;s^0,s^j_-;\ba T]} + \|\la^j\|_{[z;s^0,s^j_-;\ba T]}
 \pr\lec \|\La^j-\la^j\|_{\Stz^1(s^0,s^j_-)} 
 \pn\le\|\La^j[z,s^j_-]_>-\la^j\|_{\Stz^1(I)}.}
The second term of (iii) is bounded using Lemma \ref{lem:scatxi}, \eqref{scatt decay} with the scattering of $\x_\I^j$ as $t\to\I$
\EQ{
 \|\La^j\|_{[z;s^j_+,\ba T;\ba T]}\le\|\x^j_\I\|_{[z_j;\ta,\I;\I]}=o_\ta.}

The remaining estimates are proved by induction on $j$. 
For $j=0$, $\eqref{laga before}=0$ by the definition and $s^0_-=s^0$.  
Assume \eqref{laga before} for some $j<l$ as an induction hypothesis. 
By the scattering of $\x^j_\I$ for $t\to\I$, Lemma \ref{lem:scatxi} implies 
\EQ{ \label{xI at sj+}
 \|\La^j[z,s^j_+]\|_{\ST(s^j_+,\ba T)}
 =\|\x_\I^j[z_j,\ta]\|_{\ST(\ta,\ba T-s^j)} = o_\ta.}
Combining it with \eqref{laga before}, \eqref{Ga s-2s+} and (i), using \eqref{use wave norm}, we obtain 
\EQ{ \label{xn at sj+}
 \pt\|(\x-\Ga^J)[z,s^j_+]-\la^{[j+1,J)}\|_{\ST(s^j_+,\ba T)}
 \pr\le 
  \|(\x-\Ga^J)[z,s^j_-]-\la^{[j,J)}\|_{\ST(s^j_+,\ba T)}
   +\|\La^j[z,s^j_-]-\la^j\|_{\ST(s^j_+,\ba T)}
  \prQ+\|\Ga^J[z,s^j_-]-\Ga^J[z,s^j_+]\|_{\ST(s^j_+,\ba T)} + \|\La^j[z,s^j_+]\|_{\ST(s^j_+,\ba T)}
  \prQ+\|(\x-\La^j)[z,s^j_-]-(\x-\La^j)[z,s^j_+]\|_{\ST(s^j_+,\ba T)}
 \pr\le o_\ta+\|\x-\La^j\|_{[z;s^j_-,s^j_+;\ba T]} = o_\ta.}
Restricting it and using \eqref{zej st away}, we obtain 
\EQ{
 \pt\|(\x-\Ga^J)[z,s^j_+]\|_{\ST(s^j_+,s^{j+1}_-)}
 = o_\ta.}
This and the smallness of $\Ga_n^J$ in \eqref{Ga vanish st} allow us to apply Lemma \ref{lem:over river} to the difference of $\x_n$ and $\Ga_n^J$ for large $\ta$ and large $n$, with the same soliton part $z_n$. Then the above decay of the linearized solutions leads to the estimate on the first term of (iii): 
\EQ{
 \|\x-\Ga^J\|_{[z;s^j_+,s^{j+1}_-;\ba T]}=o_\ta.}

If $k:=j+1<J$, then combining the above with \eqref{xn at sj+}, using \eqref{use wave norm}, we obtain 
\EQ{ 
 \pt \|(\x-\Ga^J)[z,s^k_-]-\la^{[k,J)}\|_{\ST(s^k_-,\ba T)}
 \pr\le \|(\x-\Ga^J)[z,s^j_+]-\la^{[k,J)}\|_{\ST(s^k_-,\ba T)}+\|\x-\Ga^J\|_{[z;s^j_+,s^k_-,\ba T]} = o_\ta,}
which is the first term of \eqref{laga before} for $k$. 
Restricting the interval to $(s^k_-,s^k)$, we may discard $\la^{[k+1,J)}$ by \eqref{zej st away}, as well as $\Ga^J[z,s^k_-]$ by \eqref{Ga vanish st}, where we are allowed to linearize $\Ga^J$ by $[z,s^k_-]$ thanks to Lemma \ref{lem:smallsol}(II). Thus we obtain
\EQ{
 \|\x[z,s^k_-]-\la^k\|_{\ST(s^k_-,s^k)}=o_\ta.}
Since $\x(s^k_-)-\la^k(s^k_-) \weakto \x_\I^k(-\ta)-\la_\I^k(-\ta)$, by Lemma \ref{lem:Stzlim}, \eqref{profaprox}, we obtain
\EQ{
 o_\ta \pt=\|(\x_\I^k(-\ta)-\la_\I^k(-\ta))[z,s^k-\ta]\|_{\ST(s^k-\ta,s^k)}
 \pr=\|(\x_\I^k-\la_\I^k)[z_k,-\ta]\|_{\ST(-\ta,0)}.}
Taking the limit and using Lemma \ref{lem:Stzlim}, \eqref{aproxprof} with $z_k\to z_\I^k$,  
\EQ{ \label{x1-z1 fini}
 \lim_{\ta\to\I} \|\x_\I^k[z_\I^k,-\ta]-\la_\I^k\|_{\ST(-\ta,0)}=0.}
Since $\la_\I^k\in\ST(-\I,0)$ by Lemma \ref{lem:Stz}, there is $\ta_*>0$ such that 
\EQ{
 \|\la_\I^k\|_{\ST(-\I,-\ta_*)} + \sup_{\ta>\ta_*}\|\x_\I^k[z_\I^k,-\ta]-\la_\I^k\|_{\ST(-\ta,0)} \ll\cN_{1/2}^{-3}.} 
For $\ta>\ta_*$, we can apply Lemma \ref{lem:smallsol} to $\x_\I^k$ from $t=-\ta$, thereby obtain
\EQ{
 \|\x_\I^k\|_{\ST(-\ta,-\ta_*)} \le 2\|\x_\I^k[z_\I^k,-\ta]\|_{\ST(-\ta,-\ta_*)} \ll \cN_{1/2}^{-3}.}
Sending $\ta\to\I$ implies $\|\x_\I^k\|_{\ST(-\I,-\ta_*)}<\I$, so by Lemma \ref{lem:scatxi}, $\x_\I^k$ scatters with $z_\I^k$ as $t\to-\I$. Hence there exists $\fy_-^k\in H^1$ such that 
\EQ{
 \lim_{\ta\to\I}\|\x_\I^k[z_\I^k,-\ta]-\fy_-^k[z_\I^k,0]\|_{\Stz^1(-\I,0)}= 0.}
Adding this and \eqref{x1-z1 fini} yields
\EQ{
 \|(\fy_-^k-\la_\I^k)[z_\I^k,0]\|_{\ST(-\I,0)}=0,}
which implies $\fy_-^k=\la_\I^k(0)$, hence $\|\x_\I^k-\la_\I^k\|_{\Stz^1(-\I,-\ta)}\to 0$ as $\ta\to\I$. Thus we obtain (iv). 
Since $\x_\I^k=\La^k(t+s^k)$ and $\la_\I^k=\la^k(t+s^k)+o(1)$ in $\Stz^1(I-s^k)$, we obtain 
\EQ{
 \|\La^k[z,s^k_-]_>-\la^k\|_{\Stz^1(I)} \lec \|\La^k-\la^k\|_{\Stz^1(s^0,s^k_-)}=o_\ta,}
which is the second term of \eqref{laga before} for $k=j+1$, hence the induction is complete, finishing the proof for (ii)-(iv). 

Since the profiles $\La^j$ and the remainder $\Ga^J$ are vanishing $o_\ta$ in each other intervals, we obtain, using the subadditivity: Lemma \ref{lem:subadd}, as well as the monotonicity \eqref{[z] mono}, 
\EQ{ 
 \pt\|\x-\La^{[0,\ell]}-\Ga^J\|_{[z;s,s^l_+;\ba T]}
 \pr\le 
  \sum_{0\le j\le\ell} \|\x-\La^j\|_{[z;s^j_-,s^j_+;\ba T]}
  + \sum_{0\le j<l} \|\x-\Ga^J\|_{[z;s^j_+,s^{j+1}_-;\ba T]} + o_\ta
 \pn\le o_\ta.}
Since the left hand side is non-decreasing in $\ta$, we deduce that 
\EQ{ \label{xn approx}
 \ti\x^\ell:=\La^{[0,\ell]}+\la^{(\ell,J)}+ \Ga^J \implies \|\x-\ti\x^\ell\|_{[z;s,s^l_+;\ba T]}=o(1),}
where the linearized solution $\la^{(\ell,J)}$ is added for free, thanks to \eqref{ker [z]}. Using (iv) together with \eqref{zej stz}, as well as the definition of $\Ga^J$ and $\ga^J$,  we have
\EQ{
 \ti\x^\ell(s)\pt=\la^{[0,\ell]}(s)+o(1)+\la^{(\ell,J)}(s)+\ga^J(s)=\x(s)+o(1)\IN{H^1}.} 
Hence \eqref{xn approx} with \eqref{use wave norm} implies 
\EQ{ \label{xn approx ST}
 \|\x-\ti\x^\ell\|_{\ST(s,s^l_+)} \le \|\x-\ti\x^\ell\|_{[z;s,s^l_+ ,\ba T]}+o(1) = o(1),}
so we obtain, using \eqref{zej st away} as well,  
\EQ{
 \pt\|\x\|_{\ST(s,s^l_+)} 
 \le \sum_{0\le j\le\ell} \|\La^j\|_{\ST(s,s^l_+)}+\|\Ga^J\|_{\ST(s,\ba T)}+o(1),}
where each term on the right is bounded by 
\EQ{
 \pt \|\La^0\|_{\ST(s,s^l_+)}\le \|\x^0_\I\|_{\ST(0,\I)}<\I,
 \pr 1\le j<\ell \implies \|\La^j\|_{\ST(s,s^l_+)} \le \|\x^j_\I\|_{\ST(\R)}<\I,
 \pr j=\ell=l<J \implies \|\La^j\|_{\ST(s,s^l_+)} \le \|\x^l_\I\|_{\ST(-\I,\ta)}<\I,
 \pr \|\Ga^J\|_{\ST(s,\ba T)} \le 2\|\ga^J\|_{\ST(s,\ba T)} \ll 1.}
Therefore $\x$ is bounded in $\ST(s,s^l_+)$. 
It is easily upgraded to a uniform bound in $\Stz^1(s,s^l_+)$ as follows. 
Let $s=t_0<t_1<\cdots<t_N=s^l_+$ such that $\|\x\|_{\ST(t_{a-1},t_a)}\le \de$ and $N\de\le \|\x\|_{\ST(s,s^l_+)}+1$ for some small $\de>0$. 
By the Strichartz estimate, we have for each $a$ 
\EQ{
 \|\x\|_{\Stz^1(t_a,t_{a+1})} \lec \|\x(t_a)\|_{H^1}+\|\ti N(z,\x)\|_{L^2_tH^1_{6/5}(t_a,t_{a+1})},}
and the nonlinear term is estimated as before by H\"older
\EQ{ 
 \|\ti N(z,\x)\|_{L^2_tH^1_{6/5}(t_a,t_{a+1})}
 \pt\lec \|\Phi[z]\|_{L^\I_tL^3_x}\|\x\|_{L^4_tH^1_3}\|\x\|_{\ST}+
 \|\x\|_{L^\I_tH^1_x}\|\x\|_{\ST}^2
 \pr\lec (\cN_0\de+\de^2)\|\x\|_{\Stz^1(t_a,t_{a+1})}.}
Hence choosing $\de>0$ small enough, we obtain 
\EQ{
 \|\x\|_{\Stz^1(t_a,t_{a+1})} \le C\|\x(t_a)\|_{H^1} \le C\|\x\|_{\Stz^1(t_{a-1},t_a)}}
for some absolute constant $C>1$, which leads by induction to 
\EQ{
 \|\x\|_{\Stz^1(s,s^l_+)} \le C^N\|\x(s)\|_{H^1} \le C^{C\|\x\|_{\ST(s,s^l_+)}}\cN_1}
where the right hand side is bounded as shown above. 
\end{proof}

The same argument works on the other time direction $(\U T,s)$, under the scattering assumption of $\x_j^\I$ with $z_j^\I$ as $t\to-\I$ for $\s^j=-$ and $j=0$. 
In order to consider the whole interval $(\U T,\ba T)$, we should assume the scattering of $\x_j^\I$ as $t\to\s^j\I$ for $\s^j=\pm$, and of $\x_0^\I$ as $t\to\pm\I$. 
A more precise statement is as follows. 

\begin{thm} \label{thm:NPD}
Let $s\in I\in\sI^\N$ and $C(I;H^1[\mu_p])\ni u=\Phi[z]+R[z]\x$ be a sequence of solutions for \eqref{NLSP}, written in the coordinate in Lemma \ref{lem:decop2Phi}. Suppose that $u(s)$ is bounded in $H^1_x$ and let 
\EQ{
 \x[z,s]=\sum_{0\le j<J}\la^j + \ga^J, \pq \la^j=\x[z,s][z,s^j]\II}
be the linearized profile decomposition in Lemma \ref{lem:Lprof} (for a subsequence). 
If a finite $J\le J^*$ is fixed large enough, then we have the following. 

Suppose that $\sup_n\sup_{t\in I_n'}\|u_n(t)\|_{H^1_x}<\I$ for a sequence of subintervals $I'_n\subset I_n$ satisfying $s\in I'$, and let (after passing to a further subsequence if necessary) 
\EQ{
 t\in I^j_\I:=\Cu_{n\in\N}\Ca_{m\ge n}(I'_m-s^j_m) \implies (z^j_\I,\x^j_\I)(t):=\Lim_{n\to\I}(z_n,\x_n)(s^j_n+t)} 
be the weak limit in $\C\times H^1_r$. Assume that $\x^j_\I$ scatters with $z^j_\I$ as $t\to\s\I$ for each $j<J$ and $\s\in\{+,-\}$ satisfying $\s I^j_\I\supset[0,\I)$ and $\lim\s(s-s^j)\le 0$. 

Then $\sup_n\|\x_n\|_{\Stz^1(I'_n)}<\I$. 
Moreover, for each $j<J$ and $\s\in\{+,-\}$ satisfying $0\in I^j_\I$ and $\s(s-s^j)\to\I$, $\x^j_\I$ scatters with $z^j_\I$ as $t\to\s\I$ and 
\EQ{ \label{backscat}
 \lim_{n,T\to\I}\|\x^j_\I-\la^j_n(t+s^j_n)\|_{\Stz^1(\s(T,\I)\cap(I_n'-s^j_n))}=0.}
\end{thm}

The above statement has nothing to do with the excited state energy, and it is applicable even if some nonlinear profile is not scattering, if the subintervals $I'_n$ are chosen appropriately. 
Note also that $I'_n$ can be chosen depending on the linearized profile decomposition, after fixing $J$. See the next section.

\section{Scattering below the excited energy} 
We are now ready to prove the scattering to the ground states. 
For each $\mu>0$ and $A\in\R$, let $\GS(\mu,A)$ be the totality of global solution $u$ of \eqref{NLSP} satisfying
\EQ{ \label{def GS}
 \bM(u)\le\mu,\pq \bE(u)\le A.}
Let 
\EQ{ \label{def STX}
 \pt ST(\mu,A):=\sup\{\|\x\|_{\ST(0,\I)}\mid \Phi[z]+R[z]\x\in \GS(\mu,A)\},
 \pr \X:=\{(\mu,A)\mid ST(\mu,A) <\I\}.}

Introduce the following partial orders in $\R^2$  
\EQ{ \label{def orders}
 \pt(\mu_1,A_1)\le(\mu_2,A_2) \iff \mu_1\le\mu_2 \tand A_1\le A_2,
 \pr(\mu_1,A_1)\ll(\mu_2,A_2) \iff \mu_1<\mu_2 \tand A_1<A_2.}
The definition of $\X$ implies that for any $(\mu_j,A_j)\in(0,\I)\times\R$, 
\EQ{
 (\mu_1,A_1)\le(\mu_2,A_2) \text{ and } (\mu_2,A_2)\in\X \implies (\mu_1,A_1)\in\X.}

The goal of this section is to prove that for $0<\mu\ll 1$ and $A\in\R$, 
\EQ{
 (\mu,A)\in\X \iff A<\sE_1(\mu).}
$\implies$ is trivial in the defocusing case, obvious by the excited states $\Soli_1$ in the focusing case. So the question is the $\follows$ part. 

For small $H^1$ data, we have the scattering to $\Soli_0$ by \cite{gnt}, together with a uniform bound on the Strichartz norms of $\x$ in terms of $\|u(0)\|_{H^1_x}$. In fact, Lemma \ref{lem:scatxi} implies that $H^{1/2}_x$ smallness is enough. In particular, using Lemma \ref{lem:gs} and interpolation, we deduce that $(\mu,A)\in\X$ for sufficiently small $A$ for each fixed $\mu$, and for sufficiently small $\mu$ for each fixed $A$. 
Hence $\X$ contains a neighborhood of both $\{\mu=0\}$ and $\{A=0\}$.

Suppose that there exists $(\mu_0,A_0)\in(0,\I)^2\setminus \X$ satisfying $A_0<\sE_1(\mu_0)$ and $\mu_0\ll 1$.  
Put
\EQ{ \label{def EM*}
 E_*:=\sup \{A<A_0 \mid (\mu_0,A)\in\X\}, \pq M_*:=\sup\{\mu<\mu_0 \mid (\mu,E_*)\in\X\}.} 
Then
\EQ{
 \pt 0<M_*\le \mu_0,\pq 0<E_*\le A_0<\sE_1(\mu_0)\le \sE_1(M_*),}
and $(M_*,E_*)$ is minimal on $\p\X$ in the sense that 
\EQ{
 (\mu_1,A_1)<(M_*,E_*)\ll (\mu_2,A_2) \implies (\mu_1,A_1)\in\X,\pq (\mu_2,A_2)\not\in\X.}
In particular, there is a sequence $(\R^2)^\N\ni(M,E)\to(M_*,E_*)$ and a sequence of solutions $u=\Phi[z]+R[z]\x\in\GS(M,E)$ such that 
\EQ{ \label{minimizing xn}
 M \le \mu_0+o(1),\pq E<\sE_1(M),\pq \|\x\|_{\ST(0,\I)}\to \I}
See \eqref{conve seq} for the notation of sequences without index. 
The mass-energy constraint together with Lemma \ref{lem:gs} implies that $u$ is bounded in $H^1_x$, so is $\x$, while $|z|\lec\mu_0\ll 1$. 
The linearized profile decomposition: Lemma \ref{lem:Lprof} yields 
\EQ{
 \pt \x[z,0]=\sum_{0\le j<J} \la^j + \ga^J,\pq  \la^j=\x[z,0][z,s^j]\II,}
for each $J<J^*$. Let 
\EQ{
 (z_\I^j,\x_\I^j):=\lim (z,\x)(t+s^j), \pq u_\I^j:=\Phi[z_\I^j]+R[z_\I^j]\x_\I^j}
be the weak limits, solving respectively \eqref{eq zxi} and \eqref{NLSP}. 
The weak convergence implies 
\EQ{ \label{limit ME}
 \bM(u_\I^j)\le M_*, \pq \bE(u_\I^j)\le E_*.}

Fix a finite $J\le J^*$ so large that we can use Theorem \ref{thm:NPD}. 
Since $\|\x\|_{\ST(0,\I)}\to\I$, the assumption of Theorem \ref{thm:NPD} must fail for $I':=[0,\I)^\N\ni s:=0$. 
Hence there exists $l<J$ such that $s^l\ge 0$ and $\|\x^l_\I\|_{\ST(0,\I)}=\I$. 
We may choose the minimal $l$ in the sense that $s^j-s^l\to\I$ for all $j\not=l$ satisfying $s^j\ge 0$ and $\|\x^j_\I\|_{\ST(0,\I)}=\I$. 

Then \eqref{limit ME} together with the minimality of $(M_*,E_*)$ implies that $u_\I^l$ is a minimal solution which does not scatter to $\Soli_0$ as $t\to\I$, 
\EQ{
 (M_*,E_*)=(\bM(u_\I^l),\bE(u_\I^l)),}
and so the convergence is strong in $H^1_x$ for $\x(t+s^l)\to\x_\I^l(t)$ and $u(t+s^l)\to u_\I^l(t)$. 
In particular, if $l=0$ then $u(0)\to u_\I^0(0)$ strongly in $H^1_x$. 

If $l>0$, then $s^l\to\I$ and the minimality of $l$ implies that for each $j\not=l$, either $s^j\to-\I$, $s^j-s^l\to\I$ or $\|\x^j_\I\|_{\ST(0,\I)}<\I$, thereby we can apply Theorem \ref{thm:NPD} to $I':=[0,s^l]$. 
Then by \eqref{backscat}, we have
\EQ{
 \pt M_*=\bM(u^l_\I)=\bM(\Phi[z^l_\I(-s^l)])+\bM(\la^l(0))+o(1),
 \pr E_*=\bE(u^l_\I)=\bE(\Phi[z^l_\I(-s^l)])+\bH^0(\la^l(0))+o(1),}
using that $R[z]-1$ is compact on $H^1$ and that the scattering $\x^l_\I$ is weakly vanishing in $H^1$ as $t\to-\I$. 
Then the smallness of the ground states implies 
\EQ{ \label{min ele ME}
 \bH^0(\la^l(0))\ge E_*-C\mu_0+o(1).}

The same argument as above works if the assumption $\|\x\|_{\ST(0,\I)}\to\I$ is replaced with $\|\x\|_{\ST(0,T)}\to\I$ for some sequence $T\to\I$. 
Similarly, if it is replaced with $\|\x\|_{\ST(T,0)}\to\I$ for some sequence $T\to-\I$, then the same argument works in the negative time direction.

Next we prove the precompactness of the orbit of a minimal solution. 
Henceforth, the index $n$ of sequences is made explicit in order to avoid confusion. 
Let $u=\Phi[z]+R[z]\x\in\GS(M_*,E_*)$ be a global solution satisfying 
\EQ{
 (\bM(u),\bE(u))=(M_*,E_*), \pq \|\x\|_{\ST(0,\I)}=\I.}
Then for any sequence $0<t_n\to\I$, the above argument applies to $u_n:=u(t+t_n)$ on $I_n:=(-t_n,\I)\to\R$, both with $I_n':=(-t_n,0]$ and with $I_n':=[0,\I)$, because $\|\x_n\|_{\ST(-t_n,0)}=\|\x\|_{\ST(0,t_n)}\to\I$ and $\|\x_n\|_{\ST(0,\I)}=\|\x\|_{\ST(t_n,\I)}=\I$.  

If $u_\I^0$ becomes the minimal element in either case, then $u_n(0)=u(t_n)$ is strongly convergent. 
Otherwise, we get \eqref{min ele ME} for some $l=l_0>0$ in $I_n'=(-t_n,0]$ and for another $l=l_1>0$ in $I_n'=[0,\I)$, while $u^0_\I$ is scattering to $\Soli_0$ as $t\to\pm\I$. Then $\bE(u^0_\I)$ can be negative only by the soliton component, hence $\bE(u^0_\I)\gec -\mu_0$. 
Putting them into \eqref{ME decop} yields
\EQ{
  \bE_* \ge \bE(u^0_\I)+\bH^0(\la^{l_0}(0))+\bH^0(\la^{l_1}(0))+o(1)
 \ge 2E_*-C\mu_0+o(1),}
so $E_*\lec\mu_0$, contradicting the small data scattering if $\mu_0$ is small enough. 
Hence $u(t_n)$ converges strongly in $H^1_x$ after extracting a subsequence. In other words, 
\EQ{
 \{u(t) \mid t\ge 0\} \subset H^1_r(\R^3)}
is precompact for such a minimal solution $u$. 

By Lemma \ref{lem:gs}, we have a lower bound $\bK_2(u(t))\ge\ka_*:= \ka(M_*,\sE_1(M_*)-E_*)>0$. The precompactness implies that there is $R\gg 1$ such that 
\EQ{
 \sup_{t>0}\int_{|x|>R}[|\na u|^2+|u|^2+|u|^4]dx \ll \ka_*.}
Then the saturated virial identity as in Section \ref{ss:bup} implies 
\EQ{
 \p_t\LR{R f_Ru|iu_r}> \ka_* > 0,}
for all $t>0$, which obviously contradicts the boundedness of $\LR{R f_Ru|iu_r}$ in $t>0$. 
This concludes the scattering to the ground states $\Soli_0$ in (ii) of Lemma \ref{lem:gs}, and so the proof of Theorem \ref{main}.

\appendix

\section{Decay of the potential} \label{app:dec V}
Here we prove that $V,x\cdot\na V,(x\cdot\na)^2V\in L^2+L^\I_0$ for the radial function $V(x)=V(|x|)$ implies $|V(r)|+|rV_r(r)|\to 0$ as $r\to\I$. 

Decompose $V=V_2+V_\I$ such that 
\EQ{
 \|V\|_{L^2+L^\I}\sim\|V_2\|_{L^2}+\|V_\I\|_{L^\I},\pq \lim_{R\to\I}\|V_\I\|_{L^\I(|x|>R)}=0.}
For any $\e>0$, there is $R>0$ such that $\|V_\I\|_{L^\I(|x|>R)}<\e$. Let $\chi:\R\to\R$ be a smooth function satisfying $\chi(t)=t$ for $|t|\ge 3$, $|\chi(t)|\le|t|$ and $0\le\chi'(t)\le 10$ for all $t$, and $\chi(t)=0$ for $|t|\le 2$. Put 
\EQ{
 V_{(\e)}:=\e\chi(V/\e),}
Then $V_{(\e)}(r)\not=0$ implies $|V(r)|>2\e$ and so $|V_{(\e)}(r)|\le|V(r)|\sim|V_2(r)|$ for $r>R$. Hence for $r>R$, 
\EQ{
 |V_{(\e)}(r)|^2 \pt\le 2\int_r^\I|V_{(\e)}(r)|\cdot|r\p_r V_{(\e)}(r)|\frac{r^2dr}{r^3}
 \pr\lec \|V_{(\e)}\|_{L^2(|x|>R)}\||x|^{-3}\|_{L^2\cap L^\I(|x|>R)}\|r\p_r V_{(\e)}\|_{(L^2+L^\I)(|x|>R)}
 \pr\lec \|V_2\|_{L^2(|x|>R)}R^{-3/2}\|r\p_r V\|_{L^2+L^\I}\to 0}
as $R\to\I$. Since $|V|\le |V_{(\e)}|+3\e$, we deduce that $V(x)\to 0$ as $|x|\to\I$. The same argument implies that $r\p_rV(x)\to 0$. 

\pagebreak
\section{Table of Notation}
Besides the following list, see \eqref{conve seq}--\eqref{conve lim} for notation of sequences without index. 
{\small
\begin{longtable}{l|l|l}
 \hline 
 symbols & description & defined in \\
 \hline
 $\sg$ & sign of nonlinearity & \eqref{NLSP} \\
 $H$, $V$ & Schr\"odinger operator and the potential & \eqref{NLSP}, Section \ref{ss:asm V} \\  
 $e_0$, $\phi_0$ & unique eigenvalue and ground state of $H$ & \eqref{def phi0}\\
 $\pp$ & $L^p_x$-exponent where the wave operator is bounded & Section \ref{ss:asm V} (iv) \\
 $Q$ & ground state without $V$ & \eqref{eq Q} \\
 $\bE,\bM,\bK_2$ & energy, mass and virial & \eqref{def EM}, \eqref{def K2} \\
 $\ml{\cdot}$, $\bG$, $\bH$, $\bI$ & some functionals & \eqref{def funct}, \eqref{def I} \\
 $\bE^0$, $\bH^0$, $\bK_2^0$ & functionals without $V$ & \eqref{def H0} \\ 
 $\Soli$, $\Soli_j$, $\sE_j$ & all solitons, $j$-th bound states and their energy & \eqref{def Soli}, \eqref{def Solij}, \eqref{def Ej} \\
 $(\Phi,\Om)$ & coordinates of ground states & Lemma \ref{lem:Phi} \\
 $z_b$, $D_b$, $\mu_b$ & size of the above coordinates & Lemma \ref{lem:Phi}, \eqref{def mub} \\
 $\mu_d$, $z_d$ & size of parameterized ground states & \eqref{small ground}\\
 $\mu_p$, $D_p$ & size of ground state projection & Lemma \ref{lem:decop2Phi} \\ $\mu_e$ & small mass to characterize $\sE_1$ & Proposition \ref{Soli foc}\\
 $L^p_x$, $H^s_p$ ($H^s=H^s_2$), $B^s_{p,q}$ & Lebesgue, Sobolev and Besov spaces on $\R^3$ & Section \ref{ss:nota}, \eqref{def H1r}\\
 $H^1_r$, $H^1[\mu]$ & radial Sobolev and its subset with small mass & \eqref{def H1mu}\\
 $L^p_tX(I)$ & B-space valued $L^p$ space on interval & Section \ref{ss:nota} \\ 
 $(\cdot|\cdot)$, $\LR{\cdot|\cdot}$ & inner products on $L^2(\R^3)$ & \eqref{def inner} \\
 $\cS^t_p$, $\cS'_p$ & $L^p$-preserving scaling and its generator & \eqref{def Stp}-\eqref{def S'} \\ 
 $\ka$ & lower bound on $|\bK_2|$ & Proposition \ref{Soli foc}\\

 $\cH_c[z]$, $R[z]$ & normal subspace of $\Soli_0$ and its coordinate & \eqref{def decop}, \eqref{def Rz} \\
 $B[z]$, $\ti N(z,\x)$ & terms in the equation around $\Soli_0$ & \eqref{def MN}, \eqref{def Bz}, \eqref{def tiN}\\
 $\cdot\{s\}$ & functions defined around $s$ & \eqref{def germ} \\
 $u[z,s]$ & solution of the linearized equation & \eqref{def propa}\\
 $\D f[z,s]$ & Duhamel form of the linearized equation & \eqref{def Duh}\\
 $u[z,s]_>$ & solution with nonlinearity turnoff & \eqref{def turnoff}\\
 $\Stz^s$, $\Stz^{*s}$, $\ST$ & Strichartz norms & \eqref{def Stz}\\
 $[z;T_0,T_1;T_2]$ & semi-norm to measure deviation from $u[z,T_0]$ & \eqref{def wavenorm} \\
 $\sI$ & totality of intervals & \eqref{def sI} \\
 $\SBC$ & uniformly small, bounded and continuous functions & \eqref{def SBC}--\eqref{SBC} \\
 $u[z,s]_{(\I)}$ & linearized solution with limit initial data & \eqref{def limlin}\\
 $\bH_\te$ & fractional power of $HP_c$ & \eqref{def Hte} \\
 $J^*$, $s^j$ & the number and centers of profiles & Lemma \ref{lem:Lprof}\\
 $\la^j$, $\la^I$, $\ga^J$ & linear profiles and their sum, and remainder & Lemma \ref{lem:Lprof}, \eqref{def laI}\\
 $\diff{\cdots_\pa}$ & difference & \eqref{conve diff}\\
 $\La^j$, $\Ga^J$ & nonlinear profiles and remainder & \eqref{def La}, \eqref{eq Ga}\\
 $s_\pm^j$ & times around profiles for decomposition & \eqref{def sjpm}\\
 $o_\ta$ & vanishing terms as $n\to\I$ and $\ta\to\I$ & \eqref{def ota}\\
 $\GS(\mu,A)$ & global solutions below some mass and energy & \eqref{def GS}\\
 $\X$ & mass-energy region with uniform Strichartz bound &\eqref{def STX}\\
 $(\cdot,\cdot)\le,\ll$ & product orders & \eqref{def orders}\\
 \hline
\end{longtable}
}

\end{document}